\documentclass[11pt,reqno]{amsart}

\usepackage[a4paper,margin=30mm]{geometry}
\usepackage{amsmath,amssymb}
\usepackage{cite}
\usepackage{booktabs}
\usepackage{array}
\usepackage{tabularx}
\usepackage{float}
\usepackage{hyperref}
\usepackage{amsfonts}
\usepackage{enumerate,mathtools,bbm,bm,tcolorbox,mathrsfs}
\mathtoolsset{showonlyrefs}
\usepackage{tikz}
\usepackage{tikz-cd}
\usetikzlibrary{cd}
\usetikzlibrary{arrows.meta}
\usetikzlibrary{matrix, arrows, positioning}
\usetikzlibrary{shadows}
\usetikzlibrary{arrows.meta}
\colorlet{diagramcolor}{black}
\usetikzlibrary{calc}
\usepackage{caption}
\ifdefined\pdfmapfile
\pdfmapfile{=cm.map}
\fi

\allowdisplaybreaks
\numberwithin{equation}{section}

\newtheorem{theorem}{Theorem}[section]
\newtheorem{proposition}[theorem]{Proposition}
\newtheorem{lemma}[theorem]{Lemma}
\newtheorem{corollary}[theorem]{Corollary}
\newtheorem{assumption}[theorem]{Assumption}
\newtheorem{definition}[theorem]{Definition}

\theoremstyle{remark}
\newtheorem{remark}[theorem]{Remark}

\DeclareMathOperator{\Tr}{Tr}
\DeclareMathOperator{\rank}{rank}

\newcommand{\Sp}[1]{\mathfrak S^{#1}}

\newcommand{\R}{\mathbb{R}}
\newcommand{\C}{\mathbb{C}}

\newcommand{\Z}{\mathbb{Z}}
\newcommand{\T}{\mathbb{T}}

\newcommand{\abs}[1]{\left|#1\right|}

\newcommand{\BE}{\mathbb{E}}

\newcommand{\CS}{\mathcal{S}}

\newcommand{\SC}{\mathscr{C}}

\newcommand{\SF}{\mathscr{F}}

\newcommand{\SH}{\mathscr{H}}

\newcommand{\SN}{\mathscr{N}}

\newcommand{\FS}{\mathfrak{S}}

\newcommand{\al}{\alpha}
\newcommand{\be}{\beta}
\newcommand{\ep}{\varepsilon}
\newcommand{\ga}{\gamma}

\newcommand{\ps}{\psi}
\newcommand{\ph}{\varphi}
\newcommand{\ta}{\tau}
\renewcommand{\th}{\theta}
\newcommand{\lm}{\lambda}
\newcommand{\si}{\sigma}
\newcommand{\om}{\omega}
\newcommand{\ka}{\kappa}

\newcommand{\ze}{\zeta}

\newcommand{\De}{\Delta}

\newcommand{\Ps}{\Psi}

\newcommand{\Om}{\Omega}

\newcommand{\pl}{\partial}
\newcommand{\wt}{\widetilde}

\newcommand{\Ck}[1]{\left\{#1\right\}}

\newcommand{\Dk}[1]{\left[#1\right]}
\newcommand{\K}[1]{\left(#1\right)}

\newcommand{\No}[1]{\left\| #1 \right\|}

\newcommand{\I}{\infty}

\newcommand{\tw}{\frac{1}{2}}
\newcommand{\na}{\nabla}

\newcommand{\II}{\mathbbm{1}}

\newcommand{\gs}{\gtrsim}

\renewcommand{\Re}{\operatorname{Re}}

\newcommand{\ls}{\lesssim}
\newcommand{\rg}{\rangle}
\renewcommand{\lg}{\langle}
\makeatletter
\def\l@section{\@tocline{1}{0pt}{0pt}{}{}}%

\def\l@subsection{\@tocline{2}{0pt}{2.5em}{}{}}%

\def\l@subsubsection{\@tocline{3}{0pt}{3.0em}{}{}}%

\makeatother

\title{Endpoint orthonormal Strichartz estimates}
\author[S. Hadama]{Sonae Hadama}
\address{The University of Osaka, Japan}
\email{hadama.sonae.sci@osaka-u.ac.jp}

\author[S. Toyoshima]{Shunya Toyoshima}
\address{Saitama University, Japan}
\email{s.toyoshima.845@ms.saitama-u.ac.jp}

\thanks{The authors are grateful to Professor Neal Bez for his valuable advice. The first author was supported by JSPS KAKENHI Grant Number 26KJ0203.}
\date{}

\subjclass[2020]{Primary 35B45; Secondary 47B10, 42B37, 35Q41, 35Q49.}
\keywords{Orthonormal Strichartz estimates, Strichartz estimates, Schr\"odinger equation, transport equation.}

\begin{document}
\begin{abstract}
We establish the orthonormal Strichartz estimates in the abstract Keel--Tao framework, under precisely the same hypotheses as in their original theorem. As an important consequence, we resolve the endpoint conjecture for the Schr\"odinger equation in dimensions $d\ge 2$, which was first raised by Frank, Lewin, Lieb, and Seiringer. This consequence also yields an affirmative answer to the conjecture for the free transport equation, which was stated in dual form by Bennett, Bez, Guti\'errez, and Lee.
As a further application, we also establish a refinement of the Strichartz estimate for a single function.
\end{abstract}

\maketitle

\tableofcontents
\section{Introduction}\label{sec:introduction}
In this paper, we focus on the orthonormal extension of classical Strichartz estimates.
The classical Strichartz estimate controls space-time norms of solutions to dispersive equations.
For the free Schr\"odinger equation on $\R^d$, it can be written as
\begin{equation}\label{eq:classical strichartz}
	\||e^{it\Delta}\ph|^2\|_{L^p_t(\R;L^q_x(\R^d))}
	\ls \|\ph\|_{L^2(\R^d)}^2
\end{equation}
for admissible pairs of exponents $(p,q)\in [1,\I]\times[1,\I]$ such that
\begin{equation}\label{eq:full range}
\frac2p+\frac dq = d, \quad (d,p,q)\ne (2,1,\I).
\end{equation}
Strichartz first obtained this type of estimate through Fourier restriction \cite{Strichartz1977}.
Although there is a vast literature on the theory of Strichartz estimates for a single function, we refer only to the seminal paper by Keel and Tao \cite{KeelTao1998}, which proved \eqref{eq:classical strichartz} in the full range of exponents given by \eqref{eq:full range}.

If we naively apply \eqref{eq:classical strichartz} and the triangle inequality, we only obtain
\begin{equation}\label{eq:naive}
	\bigg\|\sum_j\nu_j|e^{it\Delta}\ph_j|^2\bigg\|_{L^p_t(\R;L^q_x(\R^d))} \ls \sum_j |\nu_j|
\end{equation}
for every sequence $\nu=(\nu_j)_j$ and for every orthonormal system $(\ph_j)_j$ in $L^2(\R^d)$.
However, we did not use the orthogonality of $(\ph_j)_j$ at all to prove \eqref{eq:naive}.
We are interested in whether we can improve \eqref{eq:naive} to
\begin{equation}\label{eq:intro-free-orthonormal}
	\bigg\|\sum_j\nu_j|e^{it\Delta}\ph_j|^2\bigg\|_{L^p_t(\R;L^q_x(\R^d))}
	\ls \bigg(\sum_j |\nu_j|^\al\bigg)^{1/\al}
\end{equation}
for some $\al>1$ by using the orthogonality.

The problem of extending classical inequalities for a single function to estimates for orthonormal systems has a long history.
An important early example is the Lieb--Thirring inequality, which was first proved in \cite{LiebThirring1975}.
Lieb and Thirring used this inequality in their proof of stability of matter.
A similar extension of the Sobolev inequality to orthonormal systems was proved by Lieb in \cite{Lieb1983}.
These inequalities have been studied extensively, and giving a comprehensive review is beyond the scope of this paper.
However, we should emphasize that the study of such estimates for orthonormal systems has been closely connected with mathematical physics.

In \cite{FrankLewinLiebSeiringer2014}, Frank, Lewin, Lieb, and Seiringer first proved \eqref{eq:intro-free-orthonormal} for exponents $p,q,\al\in[1,\I]$ such that
\begin{equation}\label{eq:range 1}
	1\leq q \le 1 + \frac2d,\quad \al \le \al_*(q):= \frac{2q}{q+1}.
\end{equation}
Note that the Schatten exponent $\al_*(q)$ is optimal in the sense that \eqref{eq:intro-free-orthonormal} fails if $\al>\al_*(q)$.
In \cite{FrankSabinRestriction2017}, Frank and Sabin extended \eqref{eq:range 1} to the full range
\begin{equation}\label{eq:intro-known-sharp-range}
	1\le q < \frac{d+1}{d-1},\quad \al \le \al_*(q).
\end{equation}
The range $q<(d+1)/(d-1)$ is optimal in the sense that \eqref{eq:intro-free-orthonormal} fails for $q=(d+1)/(d-1)$ and $\al=\al_*(q)$.
In \cite{FrankLewinLiebSeiringer2014}, Frank, Lewin, Lieb, and Seiringer already identified this endpoint and observed a logarithmic divergence in their argument.
Since the strong estimate \eqref{eq:intro-free-orthonormal} fails at this endpoint, the natural question is whether the restricted-type estimate
\begin{equation}\label{eq:conjecture}
	\bigg\|\sum_j \nu_j |e^{it\De}\ph_j|^2 \bigg\|_{L^{\frac{d+1}{d}}_t(\R;L^{\frac{d+1}{d-1}}_x(\R^d))} \ls \|\nu\|_{\ell^{\frac{d+1}{d},1}}
\end{equation}
holds or not, where $\ell^{(d+1)/d,1}$ is the Lorentz quasi-norm (see \eqref{eq:ell lorentz}).
Frank, Lewin, Lieb, and Seiringer first suggested in \cite{FrankLewinLiebSeiringer2014} that \eqref{eq:conjecture} might hold.
In \cite{FrankSabinTraceIdeals2017}, Frank and Sabin mentioned this weak-type Schatten estimate as an open question.
Later, Bez, Hong, Lee, Nakamura, and Sawano explicitly stated \eqref{eq:conjecture} as a conjecture in \cite{BezHongLeeNakamuraSawano2019}.
They also proved that the restricted-type estimate \eqref{eq:conjecture} fails when $d=1$.

The orthonormal Strichartz estimates are closely related to kinetic transport equations.
A semiclassical limiting argument shows that an orthonormal Strichartz estimate for the free Schr\"odinger evolution implies the corresponding
Strichartz estimate for the spatial density of solutions to the free transport equation;
see \cite[Proposition 5.1]{BezHongLeeNakamuraSawano2019}.
The free transport analogue of \eqref{eq:conjecture} is given by
\begin{equation}\label{eq:conjecture 2}
	\No{\int_{\R^d} f_0(x-tv,v) dv}_{L^{\frac{d+1}{d}}_t(\R;L^{\frac{d+1}{d-1}}_x(\R^d))} \ls \|f_0\|_{L_{x,v}^{\frac{d+1}{d},1}},
\end{equation}
where $L^{q,r}$ is the standard Lorentz space.
The estimate \eqref{eq:conjecture 2} was conjectured in \cite[Section 3]{BennettBezGutierrezLee2014} in a dual form.

In the present paper, we prove \eqref{eq:conjecture} and \eqref{eq:conjecture 2} for all $d\ge 2$ as a corollary of a more general estimate in an abstract setting.

\subsection{Setting and main abstract result}
Throughout this paper, we use the following notation.
\begin{itemize}
	\item Let $I\subset\R$ be a non-empty interval, which can be unbounded.
	\item Let $(X,\mu)$ be a $\si$-finite measure space such that $L^2(X)$ is separable.
	\item Let $H$ be a separable complex Hilbert space.
\end{itemize}
The above conventions apply to every theorem, definition, and assumption in this paper.

We also define two kinds of Lorentz spaces.
For $\al\in[1,\I)$ and $\be\in[1,\I]$, define
\begin{equation}\label{eq:ell lorentz}
	\|\nu\|_{\ell^{\al,\be}}
	:=
	\begin{dcases}
		\bigg(\sum_{j\ge 1} \frac{1}{j} (j^{1/\al} \nu_j^*)^\be \bigg)^{1/\be} &\quad \text{if } 1\le \be < \I,\\
		\sup_{j\ge 1}j^{1/\al} \nu_j^* &\quad \text{if } \be=\I,
	\end{dcases}
\end{equation}
where $(\nu_j^*)_j$ is the rearrangement of $(|\nu_j|)_j$ in decreasing order.  
Moreover, we write $\ell^\al := \ell^{\al,\al}$.
By using the $\ell^{\al,\be}$ quasi-norm, we define the Schatten--Lorentz class as follows.
\begin{definition}[Schatten--Lorentz class]\label{def:Schatten Lorentz}
	Let $K$ be a separable complex Hilbert space.
	Let $A$ be a compact operator on $K$ and let $\nu = (\nu_j)_{j\ge 1}$ be its singular values.
	For $\al\in [1,\I)$ and $\be\in[1,\I]$, we define
	\begin{equation}\label{eq:weak-norm}
		\|A\|_{\FS^{\al,\be}(K)} := \|\nu\|_{\ell^{\al,\be}}.
	\end{equation}
	Moreover, we write $\FS^\al(K):=\FS^{\al,\al}(K)$, which is a standard Schatten class.
\end{definition}

The following two assumptions are essential in our main result.
They are the same assumptions as in the work by Keel and Tao \cite{KeelTao1998}. 
See also \cite{Nguyen2024,Hoshiya2025c,FengMondalSongWu2026}.
The first assumption is a generalization of the unitarity of $e^{it\Delta}$.
\begin{assumption}[Energy estimate assumption]\label{ass:energy}
For each $t\in I$, we assume $U(t):H \to L^2(X)$ is bounded.
For every $\ph\in H$, the map $t\mapsto U(t)\ph$ is continuous.
Moreover, 
\begin{equation}\label{eq:B}
 B:=\sup_{t\in I} \|U(t)\|_{H\to L^2(X)}<\infty.
\end{equation}
We impose $B>0$ to exclude a trivial case.
\end{assumption}

The next assumption is a generalization of the dispersive estimate of $e^{it\Delta}$. 
\begin{assumption}[Dispersive estimate assumption]\label{ass:dispersive}
There exist $\sigma>1/2$ and $C_0>0$ such that, for all $t,\ta\in I$ with $t\neq \ta$,
\begin{equation}\label{eq:dispersive}
 \|U(t)U(\ta)^*\|_{L^1(X)\to L^\infty(X)}
 \le \frac{C_0}{|t-\ta|^{\si}}.
\end{equation}
\end{assumption}

We now state the central result in this paper.
\begin{theorem}[Endpoint estimate]\label{thm:weak}
Assume Assumptions \ref{ass:energy} and \ref{ass:dispersive}.
Let 
\begin{equation}
		p=\frac{2\si+1}{2\si},\quad q=\frac{2\si+1}{2\si-1},\quad r=2\si+1.
	\end{equation}
	Then, the restricted-type orthonormal Strichartz estimate
	\begin{equation}\label{eq:endpoint}
		\bigg\|\sum_j \nu_j |U(t)\ph_j|^2 \bigg\|_{L^p(I;L^q(X))} \le C \|\nu\|_{\ell^{p,1}}
	\end{equation}
	holds for all sequences $(\nu_j)_j\in \ell^{p,1}$ and for all orthonormal systems $(\ph_j)_j$ in $H$.
	Equivalently, the weak-type Schatten estimate
	\begin{equation}\label{eq:weak main}
		\No{\int_I U(t)^* V(t) U(t)dt}_{\FS^{r,\I}(H)} \le C \|V\|_{L^r(I;L^{r/2}(X))}
	\end{equation}
	holds for all $V \in L^r(I;L^{r/2}(X))$.
	Moreover, we can take $C = C_\sigma C_0^{2/r}B^{2(r-2)/r}$, where $C_\sigma$ is a constant depending only on $\si$.
\end{theorem}

\begin{remark}
	In the estimate \eqref{eq:weak main}, we regard $\int_I U(t)^* V(t) U(t)dt$ as a linear operator on $H$ by
	\[
	H \ni \ph \mapsto \int_I U(t)^* \Big[V(t) (U(t)\ph)\Big] dt \in H.
	\]
\end{remark}

Theorem \ref{thm:weak} gives us estimates for a wide range of exponents.
To understand it visually, we introduce Figure \ref{fig:exponents}, which corresponds to \cite[Fig.\,1]{BezHongLeeNakamuraSawano2019}.

\def\sigmaval{1.75}
\def\Dleftshift{0.125}

\pgfmathsetmacro{\Dx}{(\sigmaval-1)/\sigmaval-\Dleftshift}
\pgfmathsetmacro{\Ay}{2*\sigmaval/(2*\sigmaval+1)}
\pgfmathsetmacro{\AlongDB}{1-\Ay}
\pgfmathsetmacro{\Ax}{\Dx+\AlongDB*(1-\Dx)}

\begin{figure}[htbp]
	\centering
\begin{tikzpicture}[
	scale=0.8,
	transform shape,
	x=8cm,
	y=8cm,
	draw=black,
	text=black,
	line cap=round,
	line join=round,
	every node/.style={font=\normalsize},
	frame/.style={black,line width=0.85pt},
	boundary/.style={black,line width=1.0pt},
	guide/.style={black,densely dotted,line width=0.65pt},
	point/.style={circle,fill=black,draw=black,inner sep=1.45pt}
	]
	\coordinate (O) at (0,0);
	\coordinate (Q) at (1,0);
	\coordinate (P) at (0,1);
	\coordinate (C) at (1,1);
	
	\coordinate (D) at (\Dx,1);
	\coordinate (A) at (\Ax,\Ay);
	\coordinate (B) at (1,0);
	
	% F is defined as the exact TikZ midpoint of BD.
	\coordinate (F) at ($(D)!0.5!(B)$);
	
	% Feet of the perpendiculars on the 1/q- and 1/p-axes.
	\coordinate (Dq) at (D |- O);
	\coordinate (Aq) at (A |- O);
	\coordinate (Ap) at (A -| O);
	\coordinate (Fq) at (F |- O);
	\coordinate (Fp) at (F -| O);
	
	% Square frame and coordinate axes.
	\draw[frame] (O) rectangle (C);
	\draw[frame,->] (O) -- (1.085,0)
	node[below=4pt] {$1/q$};
	\draw[frame,->] (O) -- (0,1.085)
	node[left=4pt] {$1/p$};
	
	% Boundary segment D--A--F--B.
	\draw[boundary] (D) -- (B);
	
	% Existing projection guides.
	\draw[guide] (Dq) -- (D);
	\draw[guide] (Ap) -- (A) -- (Aq);
	
	% New perpendicular projection guides through F.
	\draw[guide] (Fp) -- (F) -- (Fq);
	
	% Points and their names.
	\node[point] at (D) {};
	\node[point] at (A) {};
	\node[point] at (F) {};
	\node[point] at (B) {};
	
	\node[above left=3pt]  at (D) {$D$};
	\node[above right=3pt] at (A) {$A$};
	\node[above right=3pt] at (F) {$F$};
	\node[above right=3pt] at (B) {$B$};
	
	% Axis labels.  Staggering the horizontal labels prevents overlap.
	\node[below=7pt,fill=white,inner sep=1.5pt]
	at (Dq) {$\displaystyle\frac{\sigma-1}{\sigma}$};
	\node[below=7pt,fill=white,inner sep=1.5pt]
	at (Aq) {$\displaystyle\frac{2\sigma-1}{2\sigma+1}$};
	\node[below=7pt,fill=white,inner sep=1.5pt]
	at (Fq) {$\displaystyle\frac{\sigma}{\sigma+1}$};
	
	\node[left=8pt,fill=white,inner sep=1.5pt]
	at (Ap) {$\displaystyle\frac{2\sigma}{2\sigma+1}$};
	\node[left=8pt,fill=white,inner sep=1.5pt]
	at (Fp) {$\displaystyle\frac{\sigma}{\sigma+1}$};
	
	\node[left=5pt]  at (P) {$1$};
	\node[below=5pt] at (Q) {$1$};
\end{tikzpicture}
\caption{}
\label{fig:exponents}
\end{figure}

According to the notation in \cite{BezHongLeeNakamuraSawano2019}, we define
\begin{align}
	&[X,Y] := \Ck{(1-\th)X + \th Y : \th \in [0,1]},\quad
	(X,Y) := \Ck{(1-\th)X + \th Y : \th \in (0,1)}, \\
	&(X,Y] := \Ck{(1-\th)X +\th Y: \th \in (0,1]}, \quad 
	[X,Y) := \Ck{(1-\th)X +\th Y: \th \in [0,1)}.
\end{align}
Theorem \ref{thm:weak} gives the estimate at $A$ in Figure \ref{fig:exponents}.
However, at $B$, Assumption \ref{ass:energy} and the triangle inequality imply
\begin{equation}\label{eq:trivial}
	\bigg\|\sum_j \nu_j |U(t)\ph_j|^2 \bigg\|_{L^\I(I;L^1(X))} \ls \|\nu\|_{\ell^1}.
\end{equation}
Moreover, at $D$, if $\si>1$, by the seminal theorem of Keel and Tao in \cite{KeelTao1998}, we obtain
\[
\||U(t)\ph|^2\|_{L^1_t(I;L^{\frac{\si}{\si-1}}(X))}\ls \|\ph\|_H^2
\]
for all $\ph\in H$. Therefore, by a naive estimate with the triangle inequality, 
\begin{equation}\label{eq:Keel Tao}
	\bigg\|\sum_j \nu_j |U(t)\ph_j|^2 \bigg\|_{L^1(I;L^{\frac{\si}{\si-1}}(X))} \ls  \|\nu\|_{\ell^1}.
\end{equation}
Hence, standard interpolation arguments imply the following corollary.
See Section \ref{subsec:abstract corollary} for a detailed proof.
\begin{corollary}\label{cor:abstract}
Let $p,q,\al\in [1,\I]$ satisfy
\begin{equation}
	\frac{1}{p} + \frac{\si}{q} = \si, \quad \al\le \min\K{p,\frac{2q}{q+1}}.
\end{equation}
Then, if $(1/q,1/p)\in [B,A)$, the orthonormal Strichartz estimate
	\begin{equation}\label{eq:non endpoint}
		\bigg\|\sum_j \nu_j |U(t)\ph_j|^2 \bigg\|_{L^p(I;L^q(X))} \ls \|\nu\|_{\ell^\al}
	\end{equation}
holds for every sequence $\nu= (\nu_j)_j$ and for every orthonormal system $(\ph_j)_j$ in $H$.
Moreover, if $\si>1$ and if $(1/q,1/p) \in (A,D]$, we have \eqref{eq:non endpoint}.
\end{corollary}

\begin{remark}[Comparison with related results]\label{rmk:comparison}
	In \cite{Nguyen2024}, Nguyen studied a general class of microlocalized semiclassical propagators, and established an abstract result to obtain the orthonormal Strichartz estimates under certain general assumptions.
	Moreover, his analysis describes carefully how a variety of bounds depend on the semiclassical parameter.
	The setup of the problem studied in \cite{Nguyen2024} is slightly different from ours; however, this is still an important work in our context because, to the best of the authors' knowledge, it is the first general result to obtain the orthonormal Strichartz estimates.
		
	Later, in \cite{Hoshiya2025c}, Hoshiya proved an abstract orthonormal Strichartz theorem for spectrally localized propagators associated with a self-adjoint Hamiltonian.
	In terms of Corollary \ref{cor:abstract}, he imposed a structural assumption $U(t)=e^{-it\SH}\phi(\SH)$, where $\SH$ is a Hamiltonian.
	On the segment $[B,A)$ in Figure \ref{fig:exponents}, his result agrees with Corollary \ref{cor:abstract}.
	On the segment $[A,D]$, however, his theorem gives Schatten exponents strictly below the critical value $p$.
	In contrast, Theorem \ref{thm:weak} gives the Lorentz endpoint at $A$, and Corollary \ref{cor:abstract} reaches the critical exponent on $(A,D]$.
	Moreover, the assumptions of the present result are also more flexible, since we do not impose any structural assumptions.
	
	In \cite{FengMondalSongWu2026}, Feng, Mondal, Song, and Wu developed related theory on abstract measure spaces.
	Their non-endpoint estimates are close to estimates in \cite{Hoshiya2025c} and agree with the corresponding part of Corollary \ref{cor:abstract}.
	However, their result does not contain the Lorentz endpoint at $A$ in Figure \ref{fig:exponents} or the critical conclusion on $(A,D]$. 
		
	In summary, Nguyen provided a semiclassical model, while Hoshiya and Feng--Mondal--Song--Wu developed abstract non-endpoint theories.
	Theorem \ref{thm:weak} and Corollary \ref{cor:abstract} add the missing endpoint information under weaker structural assumptions.
\end{remark}

\subsection{Consequences}
Theorem \ref{thm:weak} and Corollary \ref{cor:abstract} yield many results, but we only give two crucial consequences and one further application.

The first crucial consequence is about the free Schr\"odinger equation.
It is very well known that the propagator $e^{it\De}$ satisfies
\begin{align}
	&\|e^{it\De}\|_{L^2(\R^d)\to L^2(\R^d)} = 1, \\
	&\|e^{it\De}(e^{i\ta\De})^*\|_{L^1(\R^d) \to L^\I(\R^d)} \ls \frac{1}{|t-\ta|^{d/2}}.
\end{align}
Thus, applying Theorem \ref{thm:weak} and Corollary \ref{cor:abstract}, we obtain the following.
\begin{corollary}\label{cor:Schr}
	Let $d\ge 2$. Assume that $p,q,\al\in [1,\I]$ satisfy $2/p+d/q=d$.
	Then, the following statements hold.
	\begin{enumerate}[$(i)$]
		\item If $1\le q < (d+1)/(d-1)$ and $1\le \al \le 2q/(q+1)$, then it holds that
		\begin{equation}\label{eq:ONS Schr}
			\bigg\|\sum_j \nu_j |e^{it\De}\ph_j|^2\bigg\|_{L^p_t(\R;L^q_x(\R^d))} \ls \|\nu\|_{\ell^\al}
		\end{equation}
		for every sequence $\nu = (\nu_j)_j$ and for every orthonormal system $(\ph_j)_j$ in $L^2_x(\R^d)$.
		\item If $q=(d+1)/(d-1)$ and $1\le \al \le p$, then it holds that
		\begin{equation}
			\bigg\|\sum_j \nu_j |e^{it\De}\ph_j|^2\bigg\|_{L^p_t(\R;L^q_x(\R^d))} \ls \|\nu\|_{\ell^{\al,1}}.
		\end{equation}
		\item If $d\ge 3$, $(d+1)/(d-1)<q\le d/(d-2)$, and $1\le \al \le p$, then the estimate \eqref{eq:ONS Schr} holds.
	\end{enumerate}
\end{corollary}

	Corollary \ref{cor:Schr} $(ii)$ is an affirmative answer to \cite[Conjecture 1.3]{BezHongLeeNakamuraSawano2019}.
	Moreover, Corollary \ref{cor:Schr} $(iii)$ is an improvement of the known results.
	When $d\ge 3$, $(d+1)/(d-1)<q< d/(d-2)$, and $1\le \al <p$, then the estimate \eqref{eq:ONS Schr} is a direct consequence of the results of Keel–Tao and Frank–Sabin \cite{KeelTao1998, FrankSabinRestriction2017}. One can find the explicit statement in \cite{BezHongLeeNakamuraSawano2019}. However, to the best of the authors' knowledge, the critical line $\al = p$ was not known before.  

The second crucial consequence is about the Strichartz estimates for free transport equations.
By combining Theorem \ref{thm:weak} and \cite[Proposition 5.1]{BezHongLeeNakamuraSawano2019}, we obtain an affirmative answer to the conjecture in \cite[Section 3]{BennettBezGutierrezLee2014}, which was stated in a dual form.
\begin{corollary}\label{cor:transport}
	Let $d\ge 2$. Then, it holds that
	\begin{equation}\label{eq:transport}
		\bigg\|\int_{\R^d} f_0(x-tv,v)dv\bigg\|_{L^{\frac{d+1}{d}}_t(\R;L^{\frac{d+1}{d-1}}_x(\R^d))} \ls \|f_0\|_{L^{\frac{d+1}{d},1}_{x,v}(\R^{2d})}
	\end{equation}
	for all $f_0 \in L^{\frac{d+1}{d},1}_{x,v}(\R^{2d})$, where $L^{q,r}$ is the standard Lorentz space.
\end{corollary}
\begin{remark}
Note that $f(t,x,v)=f_0(x-tv,v)$ is a solution to the free transport equation
	\[
	\begin{dcases}
		\pl_t f + v\cdot \na_x f = 0, \quad f:\R_t \times \R^d_x \times \R^d_v \to \R, \\
		f(0)=f_0.
	\end{dcases}
	\]
	Moreover, the quantity
	\[
	\int_{\R^d} f_0(x-tv,v)dv =: \rho_f (t,x)
	\]
	is called the \textit{density}, which plays an essential role in kinetic theory.
\end{remark}

Finally, as a further application, we give a refinement of the classical Strichartz estimates for a single function.
Let $\dot B^s_{p,q}$ be the standard homogeneous Besov space.

\begin{corollary}[Refinement of the classical Strichartz estimate]\label{cor:single-function}
	Let $d\ge 3$.
	Then the following statements hold.
	\begin{enumerate}[$(i)$]
		\item Let $p,q\in(1,\I)$ satisfy
		\begin{equation}\label{eq:single-admissible}
			\frac{2}{2p}+\frac{d}{2q}=\frac{d}{2},
			\qquad
			\frac{2(d+1)}{d-1}<2q<\frac{2d}{d-2}.
		\end{equation}
		Then, it holds that
		\begin{equation}\label{eq:single-upper-lorentz}
			\|e^{it\De}\ph\|_{L^{2p}_t(\R;L^{2q}_x(\R^d))}
			\ls
			\|\ph\|_{\dot B^0_{2,2p}(\R^d)}.
		\end{equation}
		
		\item The $\dot B^0_{2,2p}$-norm in \eqref{eq:single-upper-lorentz} is optimal.
		More precisely, let $(p,q)$ satisfy \eqref{eq:single-admissible}, and let $\ka \in[1,\I)$.  If
		\begin{equation}\label{eq:single-hypothetical}
			\|e^{it\De}\ph\|_{L^{2p}_t(\R;L^{2q}_x(\R^d))}
			\ls \|\ph\|_{\dot B^0_{2,\ka}(\R^d)}
		\end{equation}
		holds for every $\ph\in\mathcal S(\R^d)$ whose Fourier transform is
		supported away from the origin, then we need $\ka \le 2p$.
	\end{enumerate}
\end{corollary}

The passage from orthonormal estimates to refinements in Besov spaces is standard; see, for example,
\cite{FrankSabinRestriction2017,BezLeeNakamura2021}.
To the best of the authors' knowledge, the best refinement of the form \eqref{eq:single-hypothetical} in the previous works was $\ka < 2p$.
Corollary \ref{cor:single-function} $(i)$ reaches $\ka=2p$, and Corollary \ref{cor:single-function} $(ii)$ shows that this is optimal.

\subsection{Review of other related works}
Following the pioneering works \cite{FrankLewinLiebSeiringer2014, FrankSabinRestriction2017} discussed earlier, the theory of orthonormal Strichartz estimates has developed substantially.

An important early result was obtained by Chen, Hong, and Pavlovi\'c in \cite{ChenHongPavlovic2017, ChenHongPavlovic2018}.
They proved estimates of the form
\begin{equation}\label{eq:CHP}
	\bigg\||\na_x|^{1/2}\sum_j \nu_j |e^{it\De}\ph_j|^2 \bigg\|_{L^2_t(\R;H^s_x(\R^d))} \ls \|\nu\|_{\ell^2}
\end{equation}
for every orthonormal system $(\ph_j)_j$ in $H^{s_1}(\R^d)$ with appropriate $s,s_1\ge 0$.
Their method is completely different from the arguments in \cite{FrankLewinLiebSeiringer2014, FrankSabinRestriction2017} and is essentially based on the space-time Fourier transform. The estimate exhibits a half-derivative gain, which arises from essentially the same mechanism as bilinear Strichartz estimates. We should emphasize that \eqref{eq:intro-free-orthonormal} does not imply \eqref{eq:CHP} and vice versa.

In a context closer to \cite{FrankLewinLiebSeiringer2014, FrankSabinRestriction2017}, Bez, Hong, Lee, Nakamura, and Sawano considered initial data with regularity in \cite{BezHongLeeNakamuraSawano2019}. More precisely, they proved \eqref{eq:intro-free-orthonormal} for all orthonormal systems $(\ph_j)_j$ in $\dot{H}^s(\R^d)$ for a certain range of $s>0$ and for appropriate exponents $p,q,\al\in[1,\I]$.

In \cite{FrankSabinRestriction2017}, Frank and Sabin also established the orthonormal Strichartz estimates for wave and Klein--Gordon equations in a certain range of exponents. 
In \cite{BezLeeNakamura2021}, Bez, Lee, and Nakamura extended and sharpened these results through weighted oscillatory-integral estimates. They obtained optimal summability exponents in several important cases and also treated fractional Schrödinger equations.
For the wave equation, see also \cite{BezKinoshitaShirakiWave2025, BezLeeNakamuraWave2025}.

Changing the order of the mixed norm led to a different boundary theory.
In \cite{BezLeeNakamura2020}, Bez, Lee, and Nakamura proved sharp maximal-in-time and maximal-in-space estimates for the one-dimensional Schr\"odinger equation. The maximal-in-time estimates apply to Carleson's pointwise convergence problem, while the maximal-in-space estimates give a boundary case of orthonormal Strichartz estimates.
Bez, Kinoshita, and Shiraki subsequently developed this connection in higher-dimensional and positive-regularity settings in  \cite{BezKinoshitaShiraki2024}.
Some important progress on maximal estimates for orthonormal wave evolutions in low dimensions was made in \cite{KinoshitaKoShiraki2026}.

On $\T^d$, the strong dispersion effect in $\R^d$ is not available, and derivative loss becomes a part of the statement. In \cite{Nakamura2020}, Nakamura proved frequency-localized orthonormal Strichartz estimates on $\T^d$ and identified a tradeoff between Sobolev regularity and Schatten gain.

If we consider the linear Schr\"odinger equation with a time-dependent or time-independent potential $V$
\begin{equation}\label{eq:u V}
	i\pl_tu + \De u - Vu = 0,
\end{equation}
then the propagator of \eqref{eq:u V} is no longer a Fourier multiplier, so we need other tools for our analysis.
In \cite{Hadama2025}, the first author considered time-dependent potentials in the context of nonlinear problems by using a purely perturbative method.
In \cite{Hoshiya2025a}, Hoshiya developed Kato's smooth perturbation theory for Schr\"odinger operators with time-independent potentials, including several singular and magnetic potentials.
In \cite{Hoshiya2024}, many equations with potentials were studied by Kato's smooth perturbation theory and microlocal analysis.
As we have already reviewed in Remark \ref{rmk:comparison}, Hoshiya proved a Keel–Tao-type abstract theorem for spectrally localized strongly continuous unitary groups and applied it, among other examples, to repulsive Hamiltonians with logarithmically decaying perturbations \cite{Hoshiya2025c}.
Feng, Mondal, Song, and Wu developed an abstract measure-space theory for non-negative self-adjoint operators in \cite{FengMondalSongWu2026}. 
For a historical review of abstract results, see Remark \ref{rmk:comparison}.

As we briefly mentioned earlier, the orthonormal Strichartz estimates are closely connected with kinetic equations via the semiclassical limit. In \cite{Nguyen2024}, Nguyen established a general result for obtaining the semiclassically rescaled orthonormal Strichartz estimates from the dispersive estimate. In \cite{Hoshiya2025b}, Hoshiya considered the semiclassical limit of the orthonormal Strichartz estimates on scattering manifolds with applications to the Boltzmann equation.

There are several works concerning special operators. 
Orthonormal estimates for Dunkl-type operators were obtained in \cite{MondalSong2025,SenapatiBoggarapuMondal2024}, while the sharpness of the Schatten exponent was studied in \cite{GhoshSwain2024}.
For the Laguerre and special Hermite operators, see \cite{FengSong2024} and \cite{GhoshMondalSwain2024}, respectively.

Finally, we mention other related works.
In \cite{HadamaHong2025}, a new type of bound in terms of the relative entropy was obtained. 
In \cite{HadamaYamamoto2026}, some randomizations were defined and it was shown that they improve estimates in a probabilistic sense.

\section{Key ideas}\label{subsec:strategy}
In this section, we outline our proof strategy, focusing on the key ideas underlying the argument.

First of all, by the duality principle, it is enough to prove the weak-type Schatten estimate \eqref{eq:weak main}.
See Section \ref{subsec:equivalence} for a detailed argument.
After a standard reduction, we may assume that $V$ is a non-negative step function.
By Definition \ref{def:Schatten Lorentz}, it is easy to check that, for any non-negative compact operator $A$ on $K$, 
\begin{equation}
	\|A\|_{\FS^{r,\I}(K)} = \sup_{\lm>0} \lm \K{\rank\II_{[\lm,\I)}(A)}^{1/r}. 
\end{equation}
Hence, the weak-type Schatten estimate \eqref{eq:weak main} is equivalent to the following eigenvalue-counting estimate 
\begin{align}\label{eq:counting estimate I}
	\rank\II_{[\lm,\I)}\K{\int_I U(t)^* V(t) U(t) dt} \le \frac{C \|V\|^r_{L^r(I;L^{r/2}(X))}}{\lm^r}.
\end{align}
Note that if $AA^*$ is compact, then $A$ is also compact, and two operators $AA^*$ and $A^*A$ have the same nonzero eigenvalues, counted with multiplicity.
Hence, by setting $f=\sqrt{V}$, we find that \eqref{eq:counting estimate I} is equivalent to
\begin{align}\label{eq:counting estimate II}
	\rank\II_{[\lm,\I)}(fUU^* f) \le \frac{C \|V\|^{r}_{L^{r}(I;L^{r/2}(X))}}{\lm^r}, 
\end{align}
where $fUU^* f:= f(t)U(t)\int_I d\ta U(\ta)^* f(\ta)$ is a linear operator on $L^2(I\times X)$ defined by
\begin{equation}
	L^2 (I\times X) \ni u(t,x) \mapsto f(t,x)U(t)\Dk{\int_I U(\ta)^* \big(f(\ta)u(\ta)\big)d\ta}(x).
\end{equation}
In the formulation \eqref{eq:counting estimate II}, we can clearly see two time variables $(t,\ta)$, while this two-time structure is hidden in \eqref{eq:counting estimate I}.

Note that we have already assumed that $V$ is a non-negative step function,
but we can also assume $V(t,x)\in 2^\Z \cup \{0\}$ for all $(t,x)\in I\times X$ to prove the bound \eqref{eq:counting estimate I}.
Indeed, for any non-negative step function $V$, there exists another step function $\wt{V}$ such that
\[
\tw\wt{V}(t,x) \le V(t,x) \le \wt{V}(t,x), \quad \wt{V}(t,x) \in 2^{\Z}\cup \{0\}
\]
for all $(t,x) \in I\times X$.

Since $V$ is a non-negative step function and $V(t,x)\in 2^\Z\cup\{0\}$ for all $(t,x)\in I\times X$, we can write 
\[
f^2 = V = \sum_{k\in \Z} 2^k \II_{E_k},
\]
where the sets $E_k$ are pairwise disjoint.
Then,
\begin{equation}
	fUU^* f = \sum_{k,k'\in\Z} 2^{(k+k')/2} \II_{E_k} UU^* \II_{E_{k'}}.
\end{equation}
Now, let $I_{\ell,j} := [j \ell, (j+1)\ell)\cap I$.
By choosing an appropriate $\ell=\ell(k,k')$, we further decompose
\begin{align}
	fUU^* f
	&= \sum_{k,k'\in\Z} \sum_{i,j\in\Z} 2^{(k+k')/2} \II_{E_k} \II_{I_{\ell(k,k'),i}} UU^* \II_{E_{k'}} \II_{I_{\ell(k,k'),j}} \\
	&= \sum_{k,k'\in\Z} \sum_{|i-j|\ge 3} 2^{(k+k')/2} \II_{E_k} \II_{I_{\ell(k,k'),i}} UU^* \II_{E_{k'}} \II_{I_{\ell(k,k'),j}} \\
	&\quad + \sum_{k,k'\in\Z} \sum_{|i-j|\le 2} 2^{(k+k')/2} \II_{E_k} \II_{I_{\ell(k,k'),i}} UU^* \II_{E_{k'}} \II_{I_{\ell(k,k'),j}}
	=: \SF + \SN.
\end{align}
We call $\SF$ and $\SN$ the far and near parts, respectively.
From a technical point of view, the most important step is to find an appropriate $\ell(k,k')$.

The far part contains interactions between well-separated time intervals.
In this region, the singularity in the dispersive bound does not appear.
The dispersive estimate then gives a good Hilbert--Schmidt bound.
The resulting bound depends only on the space-time norm of $V$ and will ultimately be used to control the relevant eigenvalue count.

The near part contains interactions between close time intervals, so the dispersive estimate is not useful.
We therefore use the energy estimate instead.
To connect this local estimate with the eigenvalue-counting problem, we introduce two further decompositions, each serving a different purpose.
These two decompositions have different roles.

First, we divide the spectrum of $fUU^* f$ into dyadic bands by defining
\[
P_{m,\lm}:=\II_{[2^m\lm,2^{m+1}\lm)}(fUU^*f ),
\quad
N_{m,\lm}:=\rank P_{m,\lm}.
\]
Then,
\[
fUU^* f=\sum_{m\in\Z}P_{m,\lm} fUU^* fP_{m,\lm}.
\]
These projections do not localize the time variables.
Their purpose is to separate the eigenvalues of $fUU^* f$ according to their size.
In particular, $P_{0,\lm}$ selects the eigenvalues whose size is comparable to $\lm$, and $N_{0,\lm}$ is the number that we want to estimate.

Second, the near part is still a large sum over $k,k',i$, and $j$.
We regroup these terms according to the relative indices
\[
n=k-k',
\qquad
n'=j-i.
\]
Since the intervals belong to the near part, only $|n'|\leq2$ occurs.

Combining the two decompositions reduces the near-part estimate to pieces of the form
\[
P_{0,\lm}\SN_{\ep,\lm}
\bigl(n,n';P_{m,\lm} f UU^* fP_{m,\lm}\bigr)P_{0,\lm}.
\]
See Section \ref{subsec:near} for the precise definition of the above notation.
As a result, the bound is expressed in terms of the relative levels $n,n'$ and the spectral level $m$, rather than the four original indices $k,k',i,j$.

For $m\leq0$, the operator-norm estimate obtained from the energy bound can be converted into a trace estimate by multiplying it by $N_{0,\lm}$. See Lemmas \ref{lem:N operator bound} and \ref{lem:N trace bound}.
For $m\geq1$, this direct estimate is not sufficient by itself.
We therefore choose a good spectral scale $\lm_*$ for which
\[
N_{m,\lm_*}\leq 2^{1-rm}N_{0,\lm_*}
\]
for all $m\ge 1$.
The geometric decay of these ranks provides the additional decay needed to sum the higher spectral bands.
The choice of $\ell(k,k')$, the proof of the estimate for each $(n,m)$, and the final numerical summation are the technical part of the argument.
They are given in Section \ref{subsec:near}, especially Lemmas \ref{lem:N operator bound} and \ref{lem:N trace bound}, and in Lemma \ref{lem:numerical}.

After these pieces have been summed, choosing $\ep>0$ sufficiently small gives
\[
\left|\Tr_{L^2(I\times X)}(P_{0,\lm_*}\SN_{\ep,\lm_*}P_{0,\lm_*})\right|
\leq \frac12\lm_*N_{0,\lm_*}.
\]
By the definition of $P_{0,\lm_*}$, we also have
\[
\lm_*N_{0,\lm_*}
\le \abs{\Tr_{L^2(I\times X)}(P_{0,\lm_*}\SF P_{0,\lm_*})} + \abs{\Tr_{L^2(I\times X)}(P_{0,\lm_*}\SN P_{0,\lm_*})}.
\]
Thus, the near contribution can be absorbed into the spectral lower bound.
Therefore, we obtain
\[
\lm_*N_{0,\lm_*}
\le 2\abs{\Tr_{L^2(I\times X)}(P_{0,\lm_*}\SF P_{0,\lm_*})}.
\]
Finally, we apply the Hilbert--Schmidt estimate for the far part.

In summary, the proof has a simple division of roles: the energy estimate controls short-time interactions, the dispersive estimate controls long-time interactions, and the choice of a suitable spectral scale $\lambda_*$ allows these two bounds to be combined at the endpoint.
The estimate is first proved for dyadic step functions, and a density argument then gives \eqref{eq:weak main} for general $V$.
Finally, the duality principle gives \eqref{eq:endpoint}.

\section{Preliminaries}
In this section, we first fix the notation for time intervals and step functions.
We then record some basic facts about Schatten classes.
Finally, we prove that $fUU^*f$ is compact for step functions $f$.
This compactness allows us to use finite-rank spectral projections in Section \ref{sec:near-far}.
\subsection{Notation and definitions}
In the sequel, we fix the following notation.
\begin{itemize}
	\item For $\ell>0$ and $j\in\mathbb Z$, set $I_{\ell,j}=I\cap[j\ell,(j+1)\ell)$.
	\item For every $t\in I$, let $j_\ell(t)\in\Z$ be such that $t\in I_{\ell,j_\ell(t)}$.
\end{itemize}
Note that the family $(I_{\ell,j})_{j\in\Z}$ is pairwise disjoint and covers $I$.
Thus, $j_\ell(t)$ is well defined for every $t\in I$.
Moreover, we define step functions and dyadic step functions as follows.
	\begin{itemize}
		\item Let $f:I\times X \to \C$ be a measurable function. We call $f:I\times X \to \C$ a step function if $f$ takes a finite number of values and there exists a bounded interval $I_0\subset I$ and a finite-measure set $X_0 \subset X$ such that $f=0$ outside of $I_0 \times X_0$.
		\item We call $f:I\times X \to \C$ a dyadic step function if $f$ is a non-negative step function and $f(t,x)\in 2^\Z \cup \{0\}$ for all $(t,x)\in I\times X$.
	\end{itemize}
Finally, we introduce some convenient and intuitive notation.
Let $V\ge 0$ and $f=\sqrt{V}$.
Define $fU:H \to L^2(I\times X)$ by 
\begin{equation}\label{eq:fU}
	(fU)\ph(t,x)=f(t,x)(U(t)\ph)(x), \quad \ph\in H.
\end{equation}
Similarly, define $U^* f: L^2(I\times X) \to H$ by
\begin{equation}
	(U^* f)u := \int_I U(\ta)^* \big(f(\ta)u(\ta)\big)d\ta, \quad u\in L^2(I\times X). 
\end{equation}
Using the notation above, we have
\begin{align}
	fU U^* f = f(t)U(t)\int_I d\ta U(\ta)^* f(\ta),\quad 
	(f U)^* (f U)= \int_I U(t)^* V(t) U(t)dt. 
\end{align}

\subsection{Basic lemmas from operator theory}\label{subsec:lemmas}
We collect necessary tools from basic operator theory.
The following lemma is well known. See, for example, \cite[Theorem 2.11]{Simon2005}.
\begin{lemma}\label{lem:direct-hs}
	Let $\Om$ be a $\si$-finite measure space. 
	Then, $A \in \FS^2(L^2(\Om))$ if and only if $A$ has an integral kernel $A(\om,\om')\in L^2(\Om\times\Om)$.
	Moreover, it holds that
	\begin{equation}\label{eq:direct-hs}
		\|A\|_{\FS^2(L^2(\Om))} = \|A(\om,\om')\|_{L^2_{\om,\om'}(\Om\times \Om)}.
	\end{equation}
\end{lemma}

The following lemma is also well known. For example, see \cite[Proposition 1.4 (6)]{Hiai1997}.
\begin{lemma}\label{lem:AB}
	Let $K$ be a Hilbert space. Let $A,B$ be non-negative compact operators on $K$.
	Let $(a_j)_{j\ge 1}$ and $(b_j)_{j\ge 1}$ be singular values of $A$ and $B$, respectively.
	If $0\le A\le B$ as an operator inequality, we have 
	\[
	0\le a_j\le b_j
	\]
	for all $j\ge 1$.
\end{lemma}

Lemma \ref{lem:AB} will be used only in the final approximation step.
It says that a positive operator inequality also gives an inequality for every singular value.
In the next statement, $\ph_1$ and $\ph_2$ also denote the corresponding multiplication operators on $L^2(X)$.

\begin{lemma}\label{lem:dispersive-hs}
Let $T:L^1(X)\to L^\I(X)$ be bounded.
Then, for any $\ph_1,\ph_2\in L^2(X)$, it holds that
\begin{equation}\label{eq:dispersive-hs}
 \|\ph_1T\ph_2\|_{\Sp{2}(L^2(X))}
 \leq \|T\|_{L^1(X)\to L^\infty(X)}
 \|\ph_1\|_{L^2(X)}\|\ph_2\|_{L^2(X)}.
\end{equation}
\end{lemma}

\begin{proof}
	Let $T(x,y)$ be the integral kernel of $T$.
	By Lemma \ref{lem:direct-hs},
	\begin{align}
		\|\ph_1 T \ph_2\|_{\FS^2(L^2(X))}
		&= \|\ph_1(x)T(x,y)\ph_2(y)\|_{L^2_{x,y}(X\times X)} \\
		&\le \|\ph_1\|_{L^2(X)} \|\ph_2\|_{L^2(X)} \|T(x,y)\|_{L^\I_{x,y}(X\times X)}.
 	\end{align}
 	By \cite[Proof of Theorem 1.3]{ArendtBukhvalov1994}, we have $\|T(x,y)\|_{L^\I_{x,y}(X\times X)} = \|T\|_{L^1(X) \to L^\I(X)}$, which completes the proof.
\end{proof}

The following technical lemma is useful for reassembling localized pieces.
\begin{lemma}\label{lem:matching}
Let $(R_j)_j$ and $(S_j)_{j}$ be finite families of pairwise orthogonal projections on $L^2(I\times X)$.
Then, for any $r\in[1,\I)$, 
\begin{equation}\label{eq:matching-contract}
 \bigg\|\sum_j R_j A S_j \bigg\|_{\FS^r(L^2(I\times X))}\le\|A\|_{\FS^r(L^2(I\times X))}.
\end{equation}
\end{lemma}

\begin{proof}
	Let $(\varepsilon_j)_j$ be independent Rademacher variables and set
	\[
	R:=\sum_j\ep_jR_j,\quad S:=\sum_j\ep_j S_j.
	\]
	Since $\BE[\varepsilon_j\varepsilon_k]=\delta_{jk}$,
	\[
	\sum_jR_jAS_j =\BE[RA S].
	\]
	Since $|\ep_j|\le 1$, as an operator inequality on $L^2(I\times X)$,
	\[
	0\le R^*R = \sum_j R_j \le 1,
	\quad
	0\le S^*S=\sum_jS_j\le 1,
	\]
	where we used the pairwise orthogonality of $R_j$ and $S_j$, namely, the fact that $R_j R_k = S_j S_k = 0$ if $j\ne k$.
	Hence
	\[
	\|R\|_{L^2(I\times X)\to L^2(I\times X)} \le 1, \quad \|S\|_{L^2(I\times X) \to L^2(I\times X)} \le 1.
	\]
	By the triangle inequality and the ideal property,
	\begin{multline*}
		\bigg\|\sum_j R_j A S_j \bigg\|_{\FS^r(L^2(I\times X))}
		\le \BE \|R A S\|_{\FS^r(L^2(I\times X))} \\
		\le \BE \|R\|_{L^2(I\times X)\to L^2(I\times X)} \|A\|_{\FS^r(L^2(I\times X))} \|S\|_{L^2(I\times X)\to L^2(I\times X)} \\
		\le \|A\|_{\FS^r(L^2(I\times X))}.
	\end{multline*}
\end{proof}

\subsection{Compactness of $fUU^* f$}
We now prove the compactness needed for the spectral decomposition in the next section.
\begin{lemma}\label{lem:time-tail}
	Let $\sigma>1/2$.  For every $\ell>0$ and every $g\in L^2(I)$,
	\begin{equation}\label{eq:time-tail}
		\left\|\int_I \frac{\II_{\K{|j_\ell(t)-j_\ell(\ta)|\geq3}}}{|t-\ta|^{2\sigma}} g(\ta) d\ta\right\|_{L^2_t(I)}
		\le \frac{C_\si}{\ell^{2\si-1}}\|g\|_{L^2(I)}.
	\end{equation}
\end{lemma}

\begin{proof}
	If $|j_\ell(t)-j_\ell(\ta)|\geq3$, then $|t-\ta|\geq2\ell$ for every $(t,\ta)$.
	Hence, Young's inequality proves \eqref{eq:time-tail}.
\end{proof}

In the proof of Lemma \ref{lem:time-tail}, the separation condition removes the singularity at $t=\ta$.
The assumption $\sigma>1/2$ is used to make the remaining kernel integrable at infinity.

We use the next lemma only to ensure that the spectral projections introduced in Section \ref{sec:near-far} have finite rank.
However, its proof shares the same spirit as the essential part of the proof; namely, the near-far decomposition.  
\begin{lemma}\label{lem:compact}
	Let $f:I \times X \to [0,\I)$ be a step function.
	Then, the operator $f U U^* f:L^2(I\times X)\to L^2(I\times X)$ is compact.
\end{lemma}

\begin{proof}
	Fix $\ell>0$ and let $V:=|f|^2$.  Consider the near part
	\begin{equation}\label{eq:compact-near-part}
		\sum_{|i-j|\leq2}
		\II_{I_{\ell,i}} f U U^* f \II_{I_{\ell,j}}.
	\end{equation}
	For every $\ph \in H$,
	\begin{align*}
		\|\II_{I_{\ell,j}}fU\ph\|_{L^2(I\times X)}^2
		&=\int_{I_{\ell,j}}\int_X
		V(t,x)|U(t)\ph(x)|^2 d \mu(x) dt \\
		&\leq\|V\|_{L^\I(I\times X)} \int_{I_{\ell,j}}\|U(t)\ph\|_{L^2(X)}^2 d t\\
		&\leq B^2\|V\|_{L^\I(I\times X)} |I_{\ell,j}|\|\ph\|_H^2
		\leq B^2\|V\|_{L^\I(I\times X)}\ell\|\ph\|_H^2.
	\end{align*}
	Hence,
	\[
	\|\II_{I_{\ell,j}}fU\|_{H \to L^2(I\times X)} \le B \|f\|_{L^\I(I\times X)} \sqrt{\ell},
	\]
	and, for any $i,j\in\mathbb Z$,
	\begin{equation}
		\left\|\II_{I_{\ell,i}} f U U^* f \II_{I_{\ell,j}}\right\|_{L^2(I\times X) \to L^2(I\times X)}
		\leq B^2\|V\|_{L^\I(I\times X)}\ell.
	\end{equation}
	For each fixed $i\in \Z$ and for each $u\in L^2(I\times X)$,
	\begin{align*}
		\bigg\|\sum_{j,\,|i-j|\leq2}
		\II_{I_{\ell,i}} f U U^* f \II_{I_{\ell,j}} u\bigg\|_{L^2(I\times X)}^2
		&\leq 5B^4\|V\|_{L^\I(I\times X)}^2\ell^2
		\sum_{j,\,|i-j|\leq2}\|\II_{I_{\ell,j}}u\|_{L^2(I\times X)}^2.
	\end{align*}
	Therefore
	\begin{align*}
		\bigg\|\sum_{|i-j|\leq2}
		\II_{I_{\ell,i}} f U U^* f \II_{I_{\ell,j}} u\bigg\|_{L^2(I\times X)}^2
		&=\sum_i\bigg\|\sum_{j,\,|i-j|\leq2} \II_{I_{\ell,i}} f U U^* f \II_{I_{\ell,j}} u\bigg\|_{L^2(I\times X)}^2 \\
		&\le 5B^4\|V\|_{L^\I(I\times X)}^2\ell^2 \sum_{|i-j|\leq2} \|\II_{I_{\ell,j}}u\|_{L^2(I\times X)}^2\\
		&\le 25B^4\|V\|_{L^\I(I\times X)}^2\ell^2 \sum_j\|\II_{I_{\ell,j}}u\|_{L^2(I\times X)}^2\\
		&=25B^4\|V\|^2_{L^\I(I\times X)}\ell^2 \|u\|_{L^2(I\times X)}^2.
	\end{align*}
	Taking square roots gives
	\begin{equation}\label{eq:compact-near}
		\bigg\|fUU^* f - \sum_{|i-j|\ge 3}\II_{I_{\ell,i}} f U U^* f \II_{I_{\ell,j}} \bigg\|_{L^2(I\times X) \to L^2(I\times X)}
		\le 5B^2\|V\|_{L^\I(I\times X)}\ell.
	\end{equation}
	Note that
	\begin{align}
		\sum_{|i-j|\ge 3} \II_{I_{\ell,i}}(t) f(t) U(t) U(\ta)^* f(\ta) \II_{I_{\ell,j}}(\ta) 
		= \II_{(|j_\ell(t)-j_\ell(\ta)|\ge 3)} f(t)U(t)U(\ta)^*f(\ta).
	\end{align}
	Also, by Lemma \ref{lem:dispersive-hs} and Assumption \ref{ass:dispersive},
	\begin{equation}
		\|f(t)U(t)U(\ta)^*f(\ta)\|_{\FS^2(L^2(X))}^2 \le \frac{C_0^2\|V(t)\|_{L^1(X)} \|V(\ta)\|_{L^1(X)} }{|t-\ta|^{2\si}}.
	\end{equation}
	Hence, by Lemma \ref{lem:direct-hs},
	\begin{align}
		&\bigg\|\sum_{|i-j|\ge 3}\II_{I_{\ell,i}} f U U^* f \II_{I_{\ell,j}} \bigg\|_{\FS^2(L^2(I\times X))}^2 \\
		&\quad = \iint_{I\times I} \II_{(|j_\ell(t)-j_\ell(\ta)|\ge 3)}\|f(t)U(t)U(\ta)^*f(\ta)\|_{\FS^2(L^2(X))}^2 dtd\ta\\
		&\quad\leq C_0^2
		\iint_{I\times I} \frac{\II_{(|j_\ell(t)-j_\ell(\ta)|\geq3)}}{|t-\ta|^{2\si}}
		\|V(t)\|_{L^1(X)} \|V(\ta)\|_{L^1(X)}dtd\ta\\
		&\quad\leq C_0^2 \|V\|_{L^2(I;L^1(X))}
		\No{\int_I \frac{\II_{(|j_\ell(t)-j_\ell(\ta)|\geq3)}}{|t-\ta|^{2\si}}\|V(\ta)\|_{L^1(X)}d\ta}_{L^2_t(I)}
	\end{align}
	By Lemma \ref{lem:time-tail},
	\begin{align}
		&\No{\int_I \frac{\II_{(|j_\ell(t)-j_\ell(\ta)|\geq3)}}{|t-\ta|^{2\si}} \|V(\ta)\|_{L^1(X)} d\ta}_{L^2_t(I)}
		\le \frac{C_\si}{\ell^{2\si-1}} \|V\|_{L^2_t(I;L^1(X))}.
	\end{align}
	Hence, the far part is Hilbert--Schmidt.
	Therefore, letting $\ell\to 0$ in \eqref{eq:compact-near}, we conclude that $fUU^* f$ is an operator-norm limit of compact operators,
	which implies that $fUU^*f$ is also compact.
\end{proof}

\section{The near-far decomposition}\label{sec:near-far}
This section contains the main decomposition used in the proof of the eigenvalue-counting estimate (Proposition \ref{prop:count}).
The decomposition has two layers.
We first separate the points of $I\times X$ according to the size of $V$.
For each pair of size levels, we then divide the time variable into intervals whose length depends on the two levels and on the spectral scale $\lambda$.
The operator $fUU^*f$ is split according to whether the two time intervals are close to each other or separated from each other. The separated part will
be estimated by the dispersive bound. The close part will be estimated by
the energy bound together with localization to spectral bands of $fUU^*f$.
The two main outputs of the section are Lemma~\ref{lem:far}, which controls
the far part globally, and Lemma~\ref{lem:N trace bound}, which controls
the trace of the near part on a suitable spectral band.

Throughout this section, we assume that $V$ is a dyadic step function.
Then, we can write
\begin{equation}\label{eq:dyadic-step}
|f(t,x)|^2 = V(t,x)=\sum_{k\in \Z} 2^k \II_{E_k}(t,x)
\end{equation}
for all $(t,x)\in I\times X$, where $E_k=\varnothing$ except for finitely many $k$ and $(E_k)_{k\in\Z}$ is disjoint.
Moreover, there exist a bounded interval $I_0 \subset I$ and a finite-measure set $X_0 \subset X$ such that $E_k \subset I_0 \times X_0$ for all $k\in\Z$.
For $\lambda>0$, $\ep \in(0,1)$, and $k,k'\in\mathbb Z$, set
\begin{equation}\label{eq:ell-kl}
	\ell(k,k'):=\frac{\varepsilon\lambda}{B^2}
	2^{-\frac{k+k'}{2}}2^{\frac{|k-k'|}{4}}.
\end{equation}
Although $\ell(k,k')$ depends on $(\ep,\lm)$, we drop them to avoid overly heavy notation.

\begin{remark}
	The choice of \eqref{eq:ell-kl} is the technical core of this paper.
	It was chosen to satisfy \eqref{eq:ell-power} in the far part analysis and \eqref{eq:choice of ell} in the near part analysis at the same time.
\end{remark}

For $j\in\mathbb Z$, define the projection on $L^2(I\times X)$ by
\begin{equation}\label{eq:Q-klj}
	Q_{\ep,\lm}(k,k',j):=\II_{E_k}\II_{I_{\ell(k,k'),j}}. 
\end{equation}
We decompose $fUU^* f$ into near and far parts as
\begin{equation}\label{eq:near far decomposition}
	f U U^* f = \SN_{\ep,\lm} + \SF_{\ep,\lm},
\end{equation}
where $\SN_{\ep,\lm}$ is the near part and $\SF_{\ep,\lm}$ is the far part.
More precisely, $\SN_{\ep,\lm}$ and $\SF_{\ep,\lm}$ are given by
\begin{align}
	\SN_{\ep,\lm} &= \sum_{k,k'\in\mathbb Z} \sum_{|i-j|\leq2}
	Q_{\ep,\lm}(k,k',i) f U U^* f Q_{\ep,\lm}(k',k,j),\label{eq:near-part} \\
	\SF_{\ep,\lm} &= \sum_{k,k'\in\mathbb Z} \sum_{|i-j|\ge 3}
	Q_{\ep,\lm}(k,k',i) f U U^* f Q_{\ep,\lm}(k',k,j).\label{eq:far-part}
\end{align}
The words ``near'' and ``far'' refer to the indices of the time intervals.
If $|i-j|\ge3$, then $|t-\ta|\ge2\ell(k,k')$ for almost every
$t\in I_{\ell(k,k'),i}$ and $\ta\in I_{\ell(k,k'),j}$.
This removes the singularity of the dispersive kernel from the far part.
The remaining pairs, with $|i-j|\le2$, form the near part.
The symmetry $\ell(k,k')=\ell(k',k)$ means that the same time partition
is used on the two sides of each block. Consequently,
$\SN_{\varepsilon,\lambda}$ and $\SF_{\varepsilon,\lambda}$ are
self-adjoint. Neither operator needs to be non-negative.

For any $m\in\Z$ and $\lm>0$, define
\begin{equation}\label{eq:P-m}
	P_{m,\lm}=\II_{[2^m\lambda,2^{m+1}\lambda)}( fUU^*f ),
	\quad N_{m,\lm} =\rank P_{m,\lm}.
\end{equation}
By Lemma \ref{lem:compact}, the operator $fUU^*f$ is compact.
Hence, $P_{m,\lm}$ is a finite-rank operator and $0\le N_{m,\lm} <\I$.
Since
\begin{equation}
	x = \sum_{m\in\Z}\II_{[2^m\lm,2^{m+1}\lm)}(x) x \II_{[2^m\lm,2^{m+1}\lm)}(x)
\end{equation}
for all $x\ge 0$, we have
\begin{equation}\label{eq:spectral-series}
	fUU^*f =\sum_{m\in\Z}P_{m,\lm} fUU^*f P_{m,\lm}.
\end{equation}
Note that the series in \eqref{eq:spectral-series} converges in operator norm.

\subsection{Far part}\label{subsec:far}
We first estimate the time-separated part.
\begin{lemma}[Far part estimate]\label{lem:far}
	Let $r:=2\si+1$.
	The far part is of the Hilbert--Schmidt class and satisfies
	\begin{equation}\label{eq:far}
		\|\SF_{\ep,\lm}\|_{\Sp{2}(L^2(I\times X))}^2
		\le \frac{C_\sigma C_0^2B^{2(r-2)}}{\ep^{r-2}\lm^{r-2}}
		\|V\|_{L^r(I;L^{r/2}(X))}^r.
	\end{equation}
\end{lemma}

\begin{proof}
Let $E_k(t)=\{x\in X:(t,x)\in E_k\}$.
By the same argument as in the proof of Lemma \ref{lem:compact}, 
\begin{align}
 \|\SF_{\ep,\lm}\|_{\FS^2(L^2(I\times X))}^2
 &= \sum_{k,k'} \bigg\| \sum_{|i-j|\ge3} Q_{\ep,\lm}(k,k',i) f U U^* f Q_{\ep,\lm}(k',k,j) \bigg\|_{\FS^2(L^2(I\times X))}^2\\
 &\le C_0^2 \sum_{k,k'}\iint_{I\times I}\frac{\II_{(|j_{\ell(k,k')}(t)-j_{\ell(k,k')}(\ta)|\geq3)}}{|t-\ta|^{2\si}}
 2^k\mu(E_k(t))2^{k'}\mu(E_{k'}(\ta)) dtd\ta,
\end{align}
where we used
\[
\Tr\Dk{\K{Q_{\ep,\lm}(k,k',i) f U U^* f Q_{\ep,\lm}(k',k,j)}^* Q_{\ep,\lm}(\wt{k},\wt{k}',i) f U U^* f Q_{\ep,\lm}(\wt{k}',\wt{k},j)} = 0
\]
if $(k,k') \ne (\wt{k},\wt{k}')$.
By Lemma \ref{lem:time-tail} and $r=1+2\si$, we have
\begin{multline}\label{eq:far-first}
	\sum_{k,k'}\iint_{I\times I}\frac{\II_{(|j_{\ell(k,k')}(t)-j_{\ell(k,k')}(\ta)|\ge3)}}{|t-\ta|^{2\si}}
	2^k\mu(E_k(t))2^{k'}\mu(E_{k'}(\ta)) dtd\ta \\
	\le C_\sigma
	\sum_{k,k'}\frac{\|2^k\mu(E_k(t))\|_{L^2_t(I)}
		\|2^{k'}\mu(E_{k'}(t))\|_{L^2_t(I)}}{\ell(k,k')^{r-2}}.
\end{multline}
By the definition of $\ell(k,k')$, namely, by \eqref{eq:ell-kl},
\begin{align} \label{eq:ell-power}
 \frac{1}{\ell(k,k')^{r-2}}
 =\frac{B^{2(r-2)} 2^{(r-2)(k+k')/2}}{\ep^{r-2} \lm^{r-2} 2^{(r-2)|k-k'|/4}}.
\end{align}
Also,
\[
 \|2^k\mu(E_k(t))\|_{L^2_t(I)}
 =\frac{\|2^{rk/2}\mu(E_k(t))\|_{L^2_t(I)}}{2^{(r-2)k/2}}.
\]
Collecting the above estimates,
\begin{equation}\label{eq:far-convolution}
 \|\SF_{\ep,\lm}\|_{\FS^2(L^2(I\times X))}^2
 \leq
 \frac{C_\si C_0^2 B^{2(r-2)}}{\ep^{r-2}\lm^{r-2}} 
 \sum_{k,k'}
 \frac{\|2^{rk/2}\mu(E_k(t))\|_{L^2_t(I)}
 	\|2^{rk'/2}\mu(E_{k'}(t))\|_{L^2_t(I)}}{2^{(r-2)|k-k'|/4}}.
\end{equation}
Since $r>2$, the discrete kernel $2^{-(r-2)|k|/4}$ belongs to $\ell^1_k(\mathbb Z)$.
Young's inequality gives
\begin{align}
 \sum_{k,k'}
 \frac{\|2^{rk/2}\mu(E_k(t))\|_{L^2_t(I)}\|2^{rk'/2}\mu(E_{k'}(t))\|_{L^2_t(I)}}{2^{(r-2)|k-k'|/4}}
 \le C_\si \sum_k\|2^{rk/2}\mu(E_k(t))\|_{L^2_t(I)}^2.
\end{align}
Finally,
\begin{align}
 \sum_k\|2^{rk/2}\mu(E_k(t))\|_{L^2_t(I)}^2
 &=\int_I\sum_k\bigl(2^{rk/2}\mu(E_k(t))\bigr)^2 d t\leq\int_I\bigg(\sum_k2^{rk/2}\mu(E_k(t))\bigg)^2 d t\notag\\
 &=\int_I\|V(t)\|_{L^{r/2}(X)}^r dt=\|V\|_{L^r(I;L^{r/2}(X))}^r.
\end{align}
Collecting the above estimates, we obtain \eqref{eq:far}.
\end{proof}

\begin{remark}
	The proof of Lemma \ref{lem:far} uses only the dispersive estimate, the dyadic level decomposition, and the assumption $r>2$.
	It does not use the spectral projections $P_{m,\lambda}$.
	All spectral information enters through the near part.
\end{remark}

\subsection{Near part}\label{subsec:near}
Let $n\in \Z$ and $n'\in\{-2,-1,0,1,2\}$.
For an operator $A$ on $L^2(I\times X)$, set
\begin{equation}\label{eq:matching-map}
	\begin{aligned}
	\SN_{\ep,\lm}(n,n';A)
	&:=\sum_{k-k'=n} \sum_{i-j=-n'} Q_{\ep,\lm}(k,k',i)A Q_{\ep,\lm}(k',k,j)\\
	&=\sum_{k',j\in\Z} Q_{\ep,\lm}(k'+n,k',j)A Q_{\ep,\lm}(k',k'+n,j+n').
	\end{aligned}
\end{equation}
Since $\ell(k,k')=\ell(k',k)$, the near part is exactly
\begin{equation} \label{eq:near-matching}
	\SN_{\ep,\lm} = \sum_{n'=-2}^2 \SN_{\ep,\lm}(0,n';fUU^* f)
	+2\sum_{n\geq1}\sum_{n'=-2}^2 \Re \SN_{\ep,\lm}(n,n';fUU^*f),
\end{equation}
where $\Re A := (A + A^*)/2$.

The next two lemmas use only the energy estimate (Assumption \ref{ass:energy}).
\begin{lemma}\label{lem:local}
	Let $\ell>0$ and $j,k\in\mathbb Z$. Then, we have
	\begin{align} \label{eq:local}
		&\II_{E_k}\II_{I_{\ell,j}} fUU^*f 
		\II_{E_k}\II_{I_{\ell,j}}\leq B^22^k\ell\,\II_{E_k}\II_{I_{\ell,j}}
	\end{align}
	as an operator inequality on $L^2(I\times X)$.
\end{lemma}

\begin{proof}
	For $\ph\in H$, by Assumption \ref{ass:energy}, 
	\begin{align*}
		\|\II_{E_k}\II_{I_{\ell,j}}fU\ph\|_{L^2(I\times X)}^2
		&=2^k\int_{I_{\ell,j}}\int_X
		\II_{E_k}(t,x)|U(t)\ph(x)|^2 d \mu(x) d t\\
		&\leq2^k\int_{I_{\ell,j}}\|U(t)\ph\|_{L^2(X)}^2 dt\leq B^22^k\ell\|\ph\|_H^2.
	\end{align*}
	This is equivalent to \eqref{eq:local}.
\end{proof}

\begin{lemma}\label{lem:local-spectral}
	For any $k,j,m\in\Z$, $\ell>0$,	and $\lm>0$, it holds that
	\begin{equation}\label{eq:local-spectral}
		\II_{E_k}\II_{I_{\ell,j}}P_{m,\lm}\II_{E_k}\II_{I_{\ell,j}}
		\leq
		\min\left(1,\frac{B^22^k\ell}{2^m\lambda}\right)
		\II_{E_k}\II_{I_{\ell,j}}.
	\end{equation}
\end{lemma}

\begin{proof}
	Since
	\[
	\II_{[2^m\lm,2^{m+1}\lm)}(x) \le \frac{1}{2^m\lm}x
	\]
	for all $x\ge 0$, we have
	\[
	P_{m,\lm}\leq \frac{1}{2^m\lambda} fUU^*f.
	\]
	Hence, Lemma \ref{lem:local} gives \eqref{eq:local-spectral}.
\end{proof}

The next lemma is the most important step in this section from a technical point of view.
\begin{lemma}\label{lem:N operator bound}
	Let $\lm>0$. Then, for any $n\ge 0$, $m\in\Z$, and $n'\in\{-2,-1,0,1,2\}$, it holds that
	\begin{align}\label{eq:matching-op}
		&\left\|P_{0,\lm} \SN_{\ep,\lm}(n,n';P_{m,\lm} fUU^*f P_{m,\lm})P_{0,\lm}\right\|_{L^2(I\times X)\to L^2(I\times X)}
		\le 2\lm w(n,m),
	\end{align}
	where
	\begin{equation}\label{eq:w-nm}
		w(n,m):=2^m \left[ \min(1,\ep 2^{3n/4-m}) \min(1,\ep 2^{-n/4-m}) \right]^{1/2}.
	\end{equation}
\end{lemma}

\begin{proof}
For any $u,v\in L^2(I\times X)$,
\begin{multline*}
	\left|
	\left\langle u, P_{0,\lm} \SN_{\ep,\lm}(n,n';P_{m,\lm} fUU^*f P_{m,\lm})P_{0,\lm} v
	\right\rangle_{L^2(I\times X)}
	\right|\\
	\leq
	\|P_{m,\lm} fUU^*f P_{m,\lm}\|_{L^2(I\times X)\to L^2(I\times X)}
	\left(\sum_{k',j} \|P_{m,\lm} Q_{\ep,\lm}(k'+n,k',j) P_{0,\lm}u\|^2_{L^2(I\times X)}\right)^{1/2}\\ 
	\cdot \left(\sum_{k',j} \|P_{m,\lm} Q_{\ep,\lm}(k',k'+n,j+n') P_{0,\lm}v\|^2_{L^2(I\times X)}\right)^{1/2}.
\end{multline*}
By Lemma \ref{lem:local-spectral},
\begin{align}
	&\sum_{k',j} \|P_{m,\lm} Q_{\ep,\lm}(k'+n,k',j) P_{0,\lm}u\|^2_{L^2(I\times X)} \\
	&\quad = \sum_{k',j} \left\lg P_{0,\lm} u \Big| \II_{E_{k'+n}} \II_{I_{\ell(k'+n,k'),j}} P_{m,\lm} \II_{E_{k'+n}} \II_{I_{\ell(k'+n,k'),j}} P_{0,\lm} u \right\rg_{L^2(I\times X)} \\
	&\quad \le  \sum_{k',j}  \min\bigg(1,\frac{B^2 2^{k'+n} \ell(k'+n,k')}{2^m\lambda}\bigg) \left\lg P_{0,\lm} u \Big|\II_{E_{k'+n}}\II_{I_{\ell(k'+n,k'),j}} P_{0,\lm} u \right\rg_{L^2(I\times X)}.
\end{align}
Note that by \eqref{eq:ell-kl},
\begin{equation}\label{eq:choice of ell}
\begin{dcases}
	\frac{B^2 2^{k'+n}\ell(k'+n,k')}{2^m\lambda}=\varepsilon2^{3n/4-m}, \\
	\frac{B^2 2^{k'}\ell(k'+n,k')}{2^m\lambda} =\varepsilon2^{-n/4-m}.
\end{dcases}
\end{equation}
Hence, 
\begin{align}\label{eq:left sum}
	\sum_{k',j} \|P_{m,\lm} Q_{\ep,\lm}(k'+n,k',j) P_{0,\lm}u\|^2_{L^2(I\times X)}
	\le \min(1,\ep2^{3n/4-m})\|u\|_{L^2(I\times X)}^2.
\end{align}
Similarly, 
\begin{align}\label{eq:right sum}
	\sum_{k',j} \|P_{m,\lm} Q_{\ep,\lm}(k',k'+n,j+n') P_{0,\lm} v\|_{L^2(I\times X)}^2
	\leq\min(1,\ep2^{-n/4-m})\|v\|_{L^2(I\times X)}^2.
\end{align}
Also, by the definition of $P_{m,\lm}$, 
\begin{equation}\label{eq:operator bound}
\|P_{m,\lambda}fUU^*fP_{m,\lambda}\|_{L^2(I\times X)\to L^2(I\times X)}
\le 2^{m+1}\lambda.
\end{equation}
Combining this bound with \eqref{eq:left sum} and \eqref{eq:right sum}, we obtain
\begin{multline*}
	\left|
	\left\langle u,	P_{0,\lambda}\SN_{\varepsilon,\lambda} (n,n';P_{m,\lambda}fUU^*fP_{m,\lambda})P_{0,\lambda}v\right\rangle_{L^2(I\times X)} \right|\\
	\le
	2^{m+1}\lambda
	\left[\min(1,\varepsilon2^{3n/4-m})
	\min(1,\varepsilon2^{-n/4-m})\right]^{1/2}
	\|u\|_{L^2(I\times X)}\|v\|_{L^2(I\times X)}\\
	=2\lambda w(n,m)\|u\|_{L^2(I\times X)} \|v\|_{L^2(I\times X)}.
\end{multline*}
Taking the supremum over unit vectors $u$ and $v$ proves \eqref{eq:matching-op}.
\end{proof}

The preceding operator-norm estimate can be converted into a trace estimate by multiplying by the rank $N_{0,\lambda}$ of the reference band. This is sufficient for the lower bands $m\leq0$. For the higher bands $m\geq1$, we need another estimate, which is given in Lemma \ref{lem:N trace bound}.

To obtain it, first fix $\lm_0>0$ and choose a spectral scale $\lambda_*\ge\lm_0$ such that
\begin{equation}
	\sup_{\ze\ge \lm_0} \ze^r N_{0,\ze}	\le 2 \lm_*^r N_{0,\lm_*}.
\end{equation}
This is Lemma \ref{lem:N trace bound} $(i)$.
Once $\lambda_*$ is fixed, the almost-maximality of $\lm_*^r N_{0,\lm_*}$ forces the ranks of all higher bands to decrease geometrically (Lemma \ref{lem:N trace bound} $(ii)$). Combining this rank decay with a Schatten $r/2$-estimate gives the additional factor $2^{-m}$ (Lemma \ref{lem:N trace bound} $(iii)$).

\begin{lemma}\label{lem:N trace bound}
	Let $r:= 2\si + 1$. Let $\lm_0>0$. Then, there exists $\lm_*\ge \lm_0$ such that the following statements hold for any $\ep\in(0,1)$.
	\begin{enumerate}[$(i)$]
		\item It holds that
		\begin{equation}
			\sup_{\ze\ge \lm_0} \ze^r N_{0,\ze}	\le 2 \lm_*^r N_{0,\lm_*};
		\end{equation}
		\item For any $m\ge 1$, it holds that
		\begin{equation}\label{eq:N-m}
			N_{m,\lm_*} \leq2^{1-rm}N_{0,\lm_*};
		\end{equation}
		\item For any $n\ge 0$, $n'\in\{-2,-1,0,1,2\}$, and $m\in\Z$, the following trace bound is true:
		\begin{multline}
			\label{eq:trace-profile}
			\left|\Tr_{L^2(I\times X)}\Big[P_{0,\lm_*}\SN_{\ep,\lm_*}\left(n,n';P_{m,\lm_*} fUU^*f P_{m,\lm_*}\right)P_{0,\lm_*}\Big]\right| \\
			\leq C_\si \lm_* N_{0,\lm_*} 
			\begin{cases}
				w(n,m),&m\leq0,\\
				\min\{w(n,m),2^{-m}\},&m\geq1.
			\end{cases}
		\end{multline}
	\end{enumerate}

\end{lemma}
\begin{proof}
	It is easy to see that $(i)$ holds.
	By Lemma \ref{lem:N operator bound}, we have \eqref{eq:trace-profile} for $m\le 0$.
	In the sequel, we consider the case $m\ge 1$.
	By the choice of $\lm_*>0$, 
	\begin{multline}
		(2^m \lm_*)^r \rank\II_{[2^m\lm_*,2^{m+1}\lm_*)}(fUU^* f)
		\le \sup_{\ze\ge \lm_0} \ze^r \rank\II_{[\ze,2\ze)}(fUU^*f) \\
		 \le 2 \lm_*^r \rank \II_{[\lm_*,2\lm_*)}(fUU^* f).
	\end{multline}
	Hence, we obtain \eqref{eq:N-m}.
	We have 
	\begin{multline}\label{eq:two-trace-bounds}
		\left|\Tr_{L^2(I\times X)}\left[P_{0,\lm_*}\SN_{\ep,\lm_*}(n,n';P_{m,\lm_*} fUU^*f P_{m,\lm_*})P_{0,\lm_*}\right]\right|\\
		\le\min\Big(N_{0,\lm_*} \|P_{0,\lm_*}\SN_{\ep,\lm_*}(n,n';P_{m,\lm_*} fUU^*f P_{m,\lm_*})P_{0,\lm_*}\|_{L^2(I\times X) \to L^2(I\times X)}, \\
		N_{0,\lm_*}^{1-2/r}\|\SN_{\ep,\lm_*}(n,n';P_{m,\lm_*} fUU^*f P_{m,\lm_*})\|_{\FS^{r/2}(L^2(I\times X))}\Big).
	\end{multline}
	On the one hand, by Lemma \ref{lem:N operator bound},
	\begin{multline}\label{eq:direct estimate}
		N_{0,\lm_*} \|P_{0,\lm_*}\SN_{\ep,\lm_*}(n,n';P_{m,\lm_*} fUU^*f P_{m,\lm_*})P_{0,\lm_*}\|_{L^2(I\times X) \to L^2(I\times X)} \\
		\le 2\lm_* N_{0,\lm_*} w(n,m)
	\end{multline}
	On the other hand, Lemma \ref{lem:matching} and \eqref{eq:operator bound} imply
	\begin{align}
		&\left\|\SN_{\ep,\lm_*}(n,n';P_{m,\lm_*} fUU^*f P_{m,\lm_*})\right\|_{\FS^{r/2}(L^2(I\times X))} \\
		&\quad \le \left\|P_{m,\lm_*}fUU^*f P_{m,\lm_*}\right\|_{\FS^{r/2}(L^2(I\times X))} \\
		&\quad \le \|P_{m,\lm_*}\|_{\FS^{r/2}(L^2(I\times X))} \left\|P_{m,\lm_*}fUU^*f P_{m,\lm_*}\right\|_{L^2(I\times X)\to L^2(I\times X)}
		\le 2^{m+1} \lm_* N_{m,\lm_*}^{2/r}.
	\end{align}
	Hence, by \eqref{eq:N-m},
	\begin{equation}
	\left\|P_{0,\lm_*} \SN_{\ep,\lm_*}(n,n';P_{m,\lm_*} fUU^*f P_{m,\lm_*})P_{0,\lm_*}\right\|_{\FS^{r/2}(L^2(I\times X))}
	\le C_\si \lm_* N_{0,\lm_*}^{2/r} 2^{-m}.
	\end{equation} 
	Together with \eqref{eq:direct estimate}, this gives \eqref{eq:trace-profile}.
\end{proof}

\section{Proofs of the main theorem and corollaries}
The preceding section established separate estimates for the near and far parts of $fUU^*f$.  We now combine those estimates and complete the proof. No new decomposition is introduced in this section.

The argument has five stages.
We first explain the equivalence of the restricted-type Strichartz estimate and the weak-type Schatten estimate.
We then prove the eigenvalue-counting estimate for dyadic step functions.
After that, we remove the dyadic and positivity assumptions on $V$.
In the last two subsections, we give proofs of Corollaries \ref{cor:abstract} and \ref{cor:single-function}.

\subsection{Duality principle}\label{subsec:equivalence}
By the duality principle, estimates \eqref{eq:endpoint} and \eqref{eq:weak main} are equivalent.
This is a basic fact, so a proof can be found, for example, in \cite{FrankLewinLiebSeiringer2014, FrankSabinRestriction2017}.
However, for completeness of this paper, we give its proof here.
Assume that the weak-type Schatten estimate \eqref{eq:weak main} is true.
Fix a sequence $\nu=(\nu_j)_j \in \ell^{p,1}$ and an orthonormal system $(\ph_j)_j$ in $H$.
Define a compact operator $\ga_0$ on $H$ by
\[
\gamma_0 = \sum_j \nu_j |\ph_j\rg \lg \ph_j|,
\]
where $|\ph\rg \lg \ph|$ is the standard bra-ket notation.
Then, a direct calculation shows that
\begin{align}
	\int_I \int_X V(t,x)\sum_j \nu_j |U(t)\ph_j(x)|^2 d\mu(x)dt
	= \Tr_H \Dk{\int_I U(t)^*V(t)U(t)dt \ga_0}.
\end{align}
Since $r'=p$ and $(r/2)' = q$, the weak-type Schatten bound \eqref{eq:weak main} implies
\begin{multline}
	\bigg|\int_I \int_X V(t,x)\sum_j \nu_j |U(t)\ph_j(x)|^2 d\mu(x)dt\bigg| \\
	\le \No{\int_I U(t)^*V(t)U(t)dt}_{\FS^{r,\I}(H)} \|\ga_0\|_{\FS^{p,1}(H)} \\
	\le C \|V\|_{L^{p'}(I;L^{q'}(X))} \|\ga_0\|_{\FS^{p,1}(H)}, \label{eq:duality estimate}
\end{multline}
where we used 
\begin{equation}\label{eq:Schatten Lorentz Holder}
	|\Tr_H(AB)| \le \|A\|_{\FS^{r,\I}(H)} \|B\|_{\FS^{r',1}(H)}
\end{equation}
in the first inequality.
We can check \eqref{eq:Schatten Lorentz Holder} as follows: Let $(a_j)_{j\ge 1}$ and $(b_j)_{j\ge 1}$ be singular values of $A$ and $B$, respectively.
Then, we have
\begin{equation}
	|\Tr_H (AB)| \le \sum_{j\ge 1} a_j b_j \le \bigg(\sup_{j\ge 1} j^{1/r} a_j\bigg) \bigg(\sum_{j\ge 1} j^{-1+1/r'} b_j \bigg) = \|A\|_{\FS^{r,\I}(H)} \|B\|_{\FS^{r',1}(H)},
\end{equation}
where we used \cite[Theorem 1.15]{Simon2005} in the first inequality.
As a consequence, the duality argument with \eqref{eq:duality estimate} implies \eqref{eq:endpoint}.
Similarly, we can prove \eqref{eq:weak main} by assuming \eqref{eq:endpoint}.

\subsection{Eigenvalue count}\label{subsec:counting}
In Section \ref{subsec:near}, we gave a trace estimate for each pair of indices $(n,m)$.
The index $n$ measures the difference between two dyadic values of $V$, and the index $m$ measures the position of a spectral band relative to the reference band. The remaining index $n'$ has only five possible values, so it contributes only a fixed constant.

The next lemma performs the numerical summation over $n$ and $m$.
This is the step that converts the separate block estimates from Lemma \ref{lem:N trace bound} $(iii)$  into a small bound for the whole near trace.
\begin{lemma}\label{lem:numerical}
Let $w(n,m)$ be as in \eqref{eq:w-nm}.  If
$0<\ep <1$, then
\begin{equation}\label{eq:numerical}
 \sum_{n\geq0}\sum_{m\leq0}w(n,m)
 +\sum_{n\geq0}\sum_{m\geq1}
 \min(w(n,m),2^{-m})
 \leq C\ep^{1/4}.
\end{equation}
\end{lemma}

\begin{proof}
	We first consider \(m\le 0\). By the definition of \(w(n,m)\),
	\[
	\begin{aligned}
		w(n,m)
		&=2^m\left[ \min(1,\ep 2^{3n/4-m})
		            \min(1,\ep 2^{-n/4-m})\right]^{1/2} \\
		&\le 2^m (\ep 2^{-n/4-m})^{1/2}
		= \ep^{1/2}2^{m/2-n/8}.
	\end{aligned}
	\]
	Therefore,
	\begin{equation}\label{eq:n+m-}
	\sum_{n\ge 0}\sum_{m\le 0}w(n,m)
	\le \ep^{1/2}
	\bigg(\sum_{m\le 0}2^{m/2}\bigg)
	\bigg(\sum_{n\ge 0}2^{-n/8}\bigg)
	\le C\varepsilon^{1/2}.
	\end{equation}	
	We next consider \(m\ge 1\).
	Since $\ep 2^{-n/4-m}<1$, we have
	\[
	w(n,m)=\min(\ep 2^{n/4},\ep^{1/2}2^{m/2-n/8}) \le \ep^{1/2}2^{m/2-n/8}
	\]
	and therefore
	\[
	 \min(w(n,m),2^{-m}) \le \ep^{1/4} 2^{-m/4} 2^{-n/16}.
	\]
	Therefore,
	\begin{equation}\label{eq:n+m+}
	\sum_{n\ge 0}\sum_{m\ge 1}\min(w(n,m),2^{-m})
	\le \ep^{1/4}
	\bigg(\sum_{m\ge 1}2^{-m/4}\bigg)
	\bigg(\sum_{n\ge 0}2^{-n/16}\bigg)
	\le C\varepsilon^{1/4}.
	\end{equation}
	Collecting \eqref{eq:n+m-} and \eqref{eq:n+m+}, we complete the proof.
\end{proof}

\begin{remark}
	If we calculate more carefully, we can improve $\ep^{1/4}$ in the RHS of \eqref{eq:numerical} to $\ep^{1/2}$.
	However, the  exact power of $\ep$ in Lemma~\ref{lem:numerical} is not important for the rest of the proof.
	What matters is that the right-hand side uniformly tends to zero as $\ep\to 0$.
	In the proof of Proposition \ref{prop:count}, we choose $\ep$ once, depending only on $\sigma$, so that
	\[
	C_\si \ep^{1/4} \le \frac12.
	\]
	After this choice, $\ep$ remains fixed.
	In particular, there is no limiting argument in which $\ep$ is sent to zero.
\end{remark}

The next proposition is the main conclusion for dyadic step functions $V=f^2$.
Indeed, Proposition \ref{prop:count} is precisely the weak-type Schatten bound at the dyadic level.

\begin{proposition}[Eigenvalue count]\label{prop:count}
For every dyadic step function $V$ and for every $\lm>0$,
the eigenvalue-counting estimate
\begin{align} \label{eq:count}
 \rank\II_{[\lambda,\infty)}(fUU^*f)
  \leq \frac{C_\sigma C_0^2B^{2(r-2)}}{\lm^r} \|V\|_{L^r(I;L^{r/2}(X))}^r
\end{align}
holds, where $f=\sqrt{V}$.
\end{proposition}

\begin{proof}
We split $fUU^*f$ into the near and far parts according to \eqref{eq:near far decomposition}.	
Fix $\lm_0>0$, and choose $\lm_*\ge \lm_0$ as in Lemma \ref{lem:N trace bound}.
Let $n\geq0$ and $n'\in\{-2,-1,0,1,2\}$.
Then, by Lemma \ref{lem:N trace bound}, we have
	\begin{multline*}
		\bigg|\Tr_{L^2(I\times X)}\Big[P_{0,\lm_*}\SN_{\ep,\lm_*}(n,n';fUU^*f)P_{0,\lm_*}\Big]\bigg|\\
		= \bigg|\Tr_{L^2(I\times X)}\Big[P_{0,\lm_*}\SN_{\ep,\lm_*}(n,n';\sum_{m\in\Z}P_{m,\lm_*} fUU^*fP_{m,\lm_*})P_{0,\lm_*}\Big]\bigg|\\
		\le
		\sum_{m\in\Z}
		\left|
		\Tr_{L^2(I\times X)}\Big[P_{0,\lm_*}\SN_{\ep,\lm_*}(n,n';P_{m,\lm_*} fUU^*fP_{m,\lm_*})P_{0,\lm_*}\Big]
		\right|\\
		\leq C_\si\lm_* N_{0,\lm_*} 
		\bigg(
		\sum_{m\leq0}w(n,m)
		+\sum_{m\ge 1}\min(w(n,m),2^{-m})
		\bigg).
	\end{multline*}
	Hence, Lemma \ref{lem:numerical} and the decomposition \eqref{eq:near-matching} now give
	\begin{equation}\label{eq:near-small}
		|\Tr_{L^2(I\times X)}[P_{0,\lm_*}\SN_{\ep,\lm_*} P_{0,\lm_*}]| \le C_\sigma \ep^{1/4} \lm_* N_{0,\lm_*} \le \tw \lm_* N_{0,\lm_*},
	\end{equation}
    where we chose $0<\ep\le1/2$, depending only on $\sigma$, so that
	\[
	C_\si \ep^{1/4}\le \frac12.
	\]
	We keep this value of $\ep$ in the rest of the proof.
	On $\operatorname{Ran}P_{0,\lm_*}$, all eigenvalues of $fUU^*f$ belong to $[\lm_*,2\lm_*)$.
	Hence
	\[
	\Tr_{L^2(I\times X)}(P_{0,\lm_*} fUU^*fP_{0,\lm_*})\geq\lm_* N_{0,\lm_*} .
	\]
		Then, by \eqref{eq:near-small},
	\begin{equation*}
		\frac12 \lm_* N_{0,\lm_*} \le |\Tr_{L^2(I\times X)}[P_{0,\lm_*}\SF_{\ep,\lm_*} P_{0,\lm_*}]|.
	\end{equation*}
	Since
	\begin{align*}
		|\Tr_{L^2(I\times X)}[P_{0,\lm_*}\SF_{\ep,\lm_*} P_{0,\lm_*}]|
		&\le \|\SF_{\ep,\lm_*}\|_{\FS^2(L^2(I\times X))} \|P_{0,\lm_*}\|_{\FS^2(L^2(I\times X))} \\
		&=\|\SF_{\ep,\lm_*}\|_{\FS^2(L^2(I\times X))} N_{0,\lm_*}^{1/2},
	\end{align*}
	we obtain
	\begin{equation}
	N_{0,\lm_*} \le \frac{2}{\lm_*}|\Tr_{L^2(I\times X)}[P_{0,\lm_*}\SF_{\ep,\lm_*} P_{0,\lm_*}]|
	\le \frac{2}{\lm_*}\|\SF_{\ep,\lm_*}\|_{\FS^2(L^2(I\times X))} N_{0,\lm_*}^{1/2}.
	\end{equation}
	When $N_{0,\lm_*} >0$, by Lemma \ref{lem:far}, 
	\begin{align}\label{eq:N-zero}
		N_{0,\lm_*} 
		\leq \frac{4}{\lm_*^2} \|\SF_{\ep,\lm_*}\|_{\FS^2(L^2(I\times X))}^2
		\leq \frac{C_\sigma C_0^2B^{2(r-2)}}{\lm_*^r}
		\|V\|_{L^r(I;L^{r/2}(X))}^r.
	\end{align}
	The factor $\varepsilon^{2-r}$ is absorbed into $C_\sigma$, because $\varepsilon$ is a constant depending only on $\si$.
	When $N_{0,\lm_*}=0$, the estimate \eqref{eq:N-zero} is clear.
	Hence, by Lemma \ref{lem:N trace bound},
	\begin{equation}
		\sup_{\ze\ge\lm_0} \ze^r N_{0,\ze}
		\le 2\lm_*^r N_{0,\lm_*}
		\le C_\sigma C_0^2B^{2(r-2)}\|V\|_{L^r(I;L^{r/2}(X))}^r,
	\end{equation}
	which implies
	\begin{equation}\label{eq:crucial}
		N_{0,\ze}
		\le \frac{C_\sigma C_0^2B^{2(r-2)}}{\ze^r}\|V\|_{L^r(I;L^{r/2}(X))}^r,
	\end{equation}
	for all $\ze>\lm_0$.
	Since $\lm_0>0$ was arbitrary, \eqref{eq:crucial} holds for all $\ze>0$.
	Finally, the intervals
	$[2^j\lambda,2^{j+1}\lambda)$ for $j\geq0$ are pairwise disjoint and their union is $[\lm,\infty)$.  Therefore
	\begin{align*}
		\rank\II_{[\lambda,\infty)}(fUU^*f)
		&=\sum_{j\ge0}
		\rank\II_{[2^j\lambda,2^{j+1}\lambda)}(fUU^*f) \\
		&\le \sum_{j\ge0} \frac{C_\sigma C_0^2B^{2(r-2)}}{(2^j\lm)^r}
		\|V\|_{L^r(I;L^{r/2}(X))}^r \\
		&\le \frac{C_\sigma C_0^2B^{2(r-2)}}{\lm^r}
		\|V\|_{L^r(I;L^{r/2}(X))}^r,
	\end{align*}
	which proves \eqref{eq:count}.
\end{proof}

\subsection{Proof of Theorem \ref{thm:weak}}\label{sec:main-proofs}
Now let $V\geq0$ be any step function.
Then, there exists a dyadic step function $\wt{V}$ such that
\begin{equation}\label{eq:dyadic-majorant}
\tw\wt{V}\le V \le \wt{V}
\end{equation}
holds pointwise.
Hence, as an operator inequality on $H$, we have
\begin{equation}\label{eq:VV}
 0\leq U^* V U \le U^* \wt{V} U.
\end{equation}
By Lemma \ref{lem:AB}, Definition \ref{def:Schatten Lorentz}, and \eqref{eq:VV}, 
\begin{align}\label{eq:positive-step}
	\|U^* V U\|_{\FS^{r,\I}(H)} \le \|U^* \wt{V} U\|_{\FS^{r,\I}(H)}.
\end{align}
Since $\wt{V}$ is a dyadic step function, by Proposition \ref{prop:count} and the argument in Section \ref{subsec:strategy}, 
\begin{equation}\label{eq:V estimate}
	\|U^* \wt{V} U\|_{\FS^{r,\I}(H)} \le C_\sigma C_0^{2/r}B^{2(1-2/r)} \|\wt{V}\|_{L^r(I;L^{r/2}(X))}.
\end{equation}
Combining \eqref{eq:positive-step}, \eqref{eq:V estimate}, and \eqref{eq:dyadic-majorant}, we conclude that
\begin{equation}\label{eq:goal}
	\|U^* V U\|_{\FS^{r,\I}(H)}  \le C_\sigma C_0^{2/r}B^{2(1-2/r)} \|V\|_{L^r(I;L^{r/2}(X))}.
\end{equation}

By a standard argument, the estimate \eqref{eq:goal} holds for every complex-valued step function $V$.
By the density argument, we obtain \eqref{eq:goal} for all $V \in L^r(I;L^{r/2}(X))$.

\subsection{Proof of Corollary \ref{cor:abstract}}\label{subsec:abstract corollary}
The proof is based on the real interpolation method.
The following lemma is important.
\begin{lemma}[See \cite{Cwikel1974, LionsPeetre1964}]\label{lem:interpolation}
	Let $p_0,p_1\in [1,\I]$, $\th\in (0,1)$, and $1/p = (1-\th)/p_0 + \th/p_1$.
	Then, we have
	\begin{equation}
		(L^{p_0}_t(I;B_0), L^{p_1}_t(I;B_1))_{\th,p} = L^p_t(I;(B_0,B_1)_{\th,p}).
	\end{equation}
\end{lemma}

	Fix an orthonormal system $\Phi=(\ph_j)_j$ in $H$.
	For every sequence $\nu = (\nu_j)_j$, define
	\[
	T_\Phi\nu:=\sum_j \nu_j|U(t)\ph_j|^2.
	\]

	We first prove the assertion on $(B,A)$.
	Let $p_A := (2\si+1)/(2\si)$ and $q_A := (2\si+1)/(2\si-1)$.
	Let $(1/q,1/p)\in(F,A)$ and put $\be:=2q/(q+1)$.
	Choose $\th\in(0,1)$ so that
	\[
	\frac1p=\frac{1-\th}{p_A} + \frac{\th}{\I},
	\quad
	\frac1q=\frac{1-\th}{q_A}+\frac{\th}{1}.
	\]
	Then,
	\[
	\frac1\be = \frac{1-\th}{p_A} + \frac{\th}{1}.
	\]
	Theorem \ref{thm:weak} and the estimate \eqref{eq:trivial} give the boundedness of
	\begin{align}
		&T_\Phi:\ell^{p_A,1}\to L^{p_A}(I;L^{q_A}(X)), \label{eq:endpoint boundedness} \\
		&T_\Phi:\ell^1\to L^\I(I;L^1(X)). \label{eq:infty 1}
	\end{align}
	By Lemma \ref{lem:interpolation}, we have
	\begin{align*}
		&(\ell^{p_A,1},\ell^1)_{\th,p}
		=\ell^{\beta,p},\\
		&\bigl(L^{p_A}(I;L^{q_A}(X)),L^\I(I;L^1(X))\bigr)_{\th,p}
		=L^p(I;L^{q,p}(X)).
	\end{align*}
	Consequently, by real interpolation, 
	\[
	\|T_\Phi\nu\|_{L^p(I;L^{q,p}(X))}
	\ls
	\|\nu\|_{\ell^{\beta,p}}.
	\]
	Since $(1/q,1/p)\in(A,F)$, we have $p<q$.
	Moreover, $p_A\le p$ gives $\beta\le p$. Hence,
	\[
	\ell^\beta\hookrightarrow\ell^{\beta,p},
	\qquad
	L^{q,p}(X)\hookrightarrow L^q(X),
	\]
	and therefore
	\begin{equation}\label{eq:AF-strong-bound}
		\|T_\Phi\nu\|_{L^p(I;L^q(X))}
		\ls
		\|\nu\|_{\ell^\beta}
	\end{equation}
	for every point in $(F,A)$.
	We now fix one point $G=(1/q_0,1/p_0)\in(A,F)$ and let $\be_0 := 2q_0/(q_0+1)$.
	By \eqref{eq:AF-strong-bound},
	\begin{equation}\label{eq:p0 q0}
		T_\Phi:\ell^{\beta_0}\to L^{p_0}(I;L^{q_0}(X))
	\end{equation}
	is bounded.
	By complex interpolation between \eqref{eq:infty 1} and \eqref{eq:p0 q0}, we conclude that \eqref{eq:AF-strong-bound} holds for every point in $(B,G)$. Since $(B,G)\cup (F,A) = (B,A)$, we obtain \eqref{eq:AF-strong-bound} on $(B,A)$.

	We next assume that $\si>1$ and $(1/q,1/p)\in (A,D)$.
	Let $q_D := \si/(\si-1)$.
	Choose $\th\in(0,1)$ such that
	\[
	\frac1p = \frac{1-\th}{p_A} + \frac{\th}{1}, \quad 
	\frac1q = \frac{1-\th}{q_A} + \frac{\th}{q_D}.
	\]
	Then, by Lemma \ref{lem:interpolation},
	\begin{align*}
		&(\ell^{p_A,1},\ell^1)_{\th,p}
		=\ell^p,\\
		&\bigl(L^{p_A}(I;L^{q_A}(X)),L^1(I;L^{q_D}(X))\bigr)_{\th,p}
		=L^p(I;L^{q,p}(X)).
	\end{align*}
	By \eqref{eq:Keel Tao},
		\begin{equation}\label{eq:D-operator-bound}
			T_\Phi:\ell^1\to L^1(I;L^{q_D}(X))
		\end{equation}
	is bounded. 
	Interpolating \eqref{eq:endpoint boundedness} and \eqref{eq:D-operator-bound},
		\[
		\|T_\Phi\nu\|_{L^p(I;L^{q,p}(X))}\ls\|\nu\|_{\ell^p}.
		\]
		Since $p\le q$ on the segment $(A,D)$, the Lorentz-space inclusion gives
		\[
		\|T_\Phi\nu\|_{L^p(I;L^q(X))}\ls\|\nu\|_{\ell^p}.
		\]
	Since $q>q_A$, we have $\min(p,2q/(q+1))=p$ on $(A,D)$.
	Collecting the above arguments, we complete the proof of Corollary \ref{cor:abstract}.

	\subsection{Proof of Corollary \ref{cor:single-function}}\label{subsec:single-function}
	The refinement estimate \eqref{eq:single-lower-bound} follows from the standard Littlewood--Paley argument; see, for example,
	\cite{FrankSabinRestriction2017,BezLeeNakamura2021}.
	We prove the optimality assertion.
	Suppose that, for some $\ka\in[1,\I)$, the estimate
	\begin{equation}\label{eq:single-hypothetical-proof}
		\|e^{it\De}\ph\|_{L^{2p}_t(\R;L^{2q}_x(\R^d))}
		\ls
		\|\ph\|_{\dot B^0_{2,\ka}(\R^d)}
	\end{equation}
	holds for every $\ph\in\CS(\R^d)$ whose Fourier transform is
	supported away from the origin.
	We show that this forces $\ka\le 2p$.
	
	Fix a nonzero function $\ps\in\CS(\R^d)$ whose Fourier transform
	is supported in a sufficiently small annulus.
	For $k\in\Z$, set $\ps_k(x):=2^{kd/2}\ps(2^k x)$.
	Choose a bounded interval $I_0\subset\R$ of positive length such that
	\[
	\SC:=\|e^{it\De}\ps\|_{L^{2p}_t(I_0;L^{2q}_x)}
	>0.
	\]
	By a simple scaling argument, 
	\begin{equation*}
		\|e^{it\De}\ps_k\|_{L^{2p}_t(2^{-2k}I_0;L^{2q}_x)}=\SC,
	\end{equation*}
	where we used $2/(2p) + d/(2q)=d/2$.
	Fix also an integer $L\ge1$ sufficiently large and set $k_j:=jL$ for $j=1,\dots,M$.
	The Fourier supports of the functions $\ps_{k_j}$ are then contained
	in mutually separated dyadic annuli.
	
	Next, choose $t_1,\ldots,t_M\in\R$ so that the intervals $I_j:=t_j+2^{-2k_j}I_0$ for $j=1,\dots,M$ are pairwise disjoint.
	For arbitrary coefficients $c_1,\ldots,c_M\in\C$, define
	\[
	\Ps_M
	:=
	\sum_{j=1}^M
	c_j e^{-it_j\De}\ps_{k_j}.
	\]
	Since the sum is finite, $\Ps_M$ is a Schwartz function.
	Moreover, its Fourier transform is supported away from the origin.
	
	Note that for every $t\in\R$ and $j$,
	\[
	\|e^{it\De}\Ps_M\|_{L^{2q}_x}
	\gs
	|c_j|
	\|e^{i(t-t_j)\De}\ps_{k_j}\|_{L^{2q}_x}.
	\]
	Since the intervals $I_j$ are pairwise disjoint, it follows that
	\begin{align}
		\|e^{it\De}\Ps_M\|_{L^{2p}_t(\R;L^{2q}_x)}^{2p}
		&\ge \sum_{j=1}^M \|e^{it\De}\Ps_M\|_{L^{2p}_t(I_j;L^{2q}_x)}^{2p}\\
		&\gs \sum_{j=1}^M |c_j|^{2p} \|e^{i(t-t_j)\De}\ps_{k_j}\|_{L^{2p}_t(I_j;L^{2q}_x)}^{2p}
		= \SC^{2p} \sum_{j=1}^M |c_j|^{2p}.\label{eq:single-lower-bound}
	\end{align}
	Thus,
	\begin{equation}\label{eq:single-lower-ell}
		\|e^{it\De}\Ps_M\|_
		{L^{2p}_t(\R;L^{2q}_x)}
		\gs \SC \|c\|_{\ell^{2p}}.
	\end{equation}
	By the definition of the Besov norm, 
	\begin{equation}\label{eq:single-besov-upper}
		\|\Ps_M\|_{\dot B^0_{2,\ka}}
		\ls	\|\ps\|_{L^2} \|c\|_{\ell^\ka}.
	\end{equation}
	Applying \eqref{eq:single-hypothetical-proof} to $\ph=\Ps_M$ and combining
	\eqref{eq:single-lower-ell} with
	\eqref{eq:single-besov-upper}, we obtain
	\[
	\|c\|_{\ell^{2p}}
	\ls
	\|c\|_{\ell^\ka},
	\]
	where the implicit constant is independent of $M$ and $c_1,\dots,c_M$.
	Therefore, we conclude that $\ka\le 2p$ is necessary.

\end{document}